\documentclass[lettersize,journal]{IEEEtran}
\usepackage{cite}
\usepackage{amsmath,amssymb,amsfonts,bm}
\usepackage{algorithmic}
\usepackage[ruled,vlined,onelanguage]{algorithm2e}
\usepackage{graphicx}
\usepackage{textcomp}
\usepackage{xcolor}
\usepackage{physics}
\usepackage{subcaption}
\usepackage{balance}
\usepackage{array}
\usepackage{fancyhdr}

\newcommand{\mat}[1]{\begin{bmatrix} #1 \end{bmatrix}}
\DeclareMathOperator{\diag}{diag}
\DeclareMathOperator{\hjb}{HJB}

\usepackage{amsthm}

\newtheoremstyle{ieeeconf}
{0pt}   
{0pt}   
{\normalfont}  
{\parindent}       
{\itshape} 
{:}         
{ } 
{\thmname{#1} \thmnumber{#2}\thmnote{ (#3)}} 
\makeatletter
\renewenvironment{proof}[1][\proofname]{\par
	\pushQED{\qed}%
	\normalfont \topsep\z@
	\trivlist
	\item[\hskip\labelsep\indent
	\itshape
	#1\@addpunct{:}]\ignorespaces
}{%
	\popQED\endtrivlist\@endpefalse
}
\makeatletter

\theoremstyle{ieeeconf}

\newtheorem{definition}{Definition}
\newtheorem{assumption}{Assumption}
\newtheorem{problem}{Problem}
\newtheorem{lemma}{Lemma}
\newtheorem{theorem}{Theorem}
\newtheorem{remark}{Remark}

\usepackage{ifthen}
\newboolean{showcomments}
\setboolean{showcomments}{false} 
\ifthenelse{\boolean{showcomments}}{ 
	\newcommand\td[1]{\textcolor{red}{\textbf{TODO:} #1}} 
	\newcommand\tdd[1]{} 
	\newcommand\comment[1]{\textcolor{blue}{\textbf{Comment:} #1}} 
	\newcommand\commentd[1]{} 
	\newcommand\frage[1]{\textcolor{orange}{\textbf{Rückfrage:} #1}} 
	\newcommand\fraged[1]{}
}{ 
	\newcommand\td[1]{}
	\newcommand\tdd[1]{} 
	\newcommand\comment[1]{} 
	\newcommand\commentd[1]{} 
	\newcommand\frage[1]{}
	\newcommand\fraged[1]{}
} 

\def\BibTeX{{\rm B\kern-.05em{\sc i\kern-.025em b}\kern-.08em
    T\kern-.1667em\lower.7ex\hbox{E}\kern-.125emX}}

\begin{document}
	
\title{Offline and Online Nonlinear Inverse Differential Games with Known and Approximated Cost and Value Function Structures}
\author{Philipp Karg, Balint Varga, and Sören Hohmann
	\thanks{All authors are with the Institute of Control Systems (IRS),
		Karlsruhe Institute of Technology (KIT), 76131 Karlsruhe, Germany. Corresponding author is Philipp Karg, {\tt\small philipp.karg@kit.edu}.}
	\thanks{This work was funded by the Deutsche Forschungsgemeinschaft (DFG, German Research Foundation) - 421428832.}}

\markboth{IEEE Transactions on Cybernetics,~Vol.~XX, No.~XX, XXXX}
{Karg \MakeLowercase{et al.}: Offline and Online Nonlinear Inverse Differential Games}

\maketitle
\thispagestyle{fancy}
\pagestyle{fancy}

\fancyhf{}
\fancyhead[CO,CE]{\copyright 2024 IEEE. This paper has been submitted for possible publication in the IEEE Transactions on Cybernetics.}

\begin{abstract}
	In this work, we propose novel offline and online Inverse Differential Game (IDG) methods for nonlinear Differential Games (DG), which identify the cost functions\comment{cost function = running cost function in this paper; integral cost function(al) = value function(al) for arbitrary policy combination with starting time $t_0$} of all players from control and state trajectories constituting a feedback Nash equilibrium. The offline approach computes the sets of all equivalent cost function parameters that yield the observed trajectories. 
	Our online method is guaranteed to converge to cost function parameters of the offline calculated sets. For both methods, we additionally analyze the case where the cost and value functions are not given by known parameterized structures and approximation structures, like polynomial basis functions, need to be chosen. Here, we found that for guaranteeing a bounded error between the trajectories resulting from the offline and online IDG solutions and the observed trajectories an appropriate selection of the cost function structures is required. They must be aligned to assumed value function structures such that the coupled Hamilton-Jacobi-Bellman equations can be fulfilled. Finally, the theoretical results and the effectiveness of our new methods are illustrated with a numerical example.
\end{abstract}


\section{Introduction} \label{sec:introduction}
\tdd{Kürzungspotential (1): Introduction und Abstract}
Inverse Optimal Control (IOC) methods have gained significant research interest in the last years. Despite the early works of Kalman \cite{Kalman.1964} and coworkers where IOC deals with the question when an arbitrary control law is an optimal one, all current methods can be understood as data-based approaches. They recover from observed state and control trajectories (so-called Ground Truth (GT)), which are assumed to constitute an Optimal Control (OC) solution, the underlying cost functions\footnote{Throughout the paper, 'cost function' refers to the so-called running cost function, i.e. the integrand of the value function to be minimized in the OC problem.}. The identification of unknown agents, like humans \cite{Mombaur.2010}, is an application example of such IOC methods. Inverse Differential Games (IDG) represent the generalization of IOC to multi-agent systems, where $N$ players optimize their individual value function and mutually influence a dynamic system. IDG methods recover the cost functions of all players from the state and the control trajectories of all players, which are assumed to constitute a solution of the Differential Game (DG), e.g. a Nash equilibrium. 
An application example is the analysis of the interaction of unknown agents, like the collision avoidance behavior of humans \cite{Karg.2024}.

IDG methods can be categorized in direct and indirect approaches \cite{Molloy.2022}. Direct IDG methods minimize an error measure between the GT trajectories and the trajectories optimal w.r.t. a current cost function guess. Since for every evaluation of this error measure a DG solution needs to be calculated, such direct approaches yield high and poorly scaling computation times. Therefore, indirect methods were developed, which minimize the violation of optimality conditions that are fulfilled with the searched cost function parameters and the observed GT trajectories. For example, these indirect IDG methods can be derived from the Minimum Principle (MP) (see e.g. \cite{Molloy.2022}) or the coupled Hamilton-Jacobi-Bellman (HJB) equations (see e.g. \cite{Inga.2019b,Lian.2022,Martirosyan.2024}). In this paper, we focus on the latter methods since they promise advantages like the formulation of online learning methods and the analysis of approximation errors, which is crucial in practice when the cost or value functions are not given by known parameterized structures and approximation structures, like polynomial basis functions or neural networks, are used to approximate the true functions. In this case, approximation errors remain and the question whether a bounded error between the trajectories resulting from the IDG solution and the GT trajectories can be guaranteed arises.

The majority of works related to HJB-based IDG methods deal only with the single player, i.e. IOC, problem (see e.g. \cite{Pauwels.2014,Kamalapurkar.2018,Self.2021,Self.2022,Town.2023,Lian.2024b,Wu.2024}). They mainly propose online learning methods for nonlinear\comment{not in \cite{Self.2021,Town.2023}} OC settings with proofed convergence behavior\comment{not for final cost function parameters in \cite{Wu.2024}} and with an analysis of erroneous approximations of the value and cost function\comment{not in \cite{Wu.2024}}. However, this analysis is not separately done for value and cost function to distinguish their influences.
Furthermore, in all mentioned references the well-known non-uniqueness of IOC solutions is only insufficiently addressed. Indeed, in \cite{Kamalapurkar.2018,Self.2021,Self.2022}, excitation conditions are assumed such that the IOC solution is unique up to a scaling factor which is however often not the case. Notable exception is the approach in \cite{Town.2023} which however deals only with a linear-quadratic (LQ) OC problem.
Regarding multi-player HJB-based IDG methods (see e.g. \cite{Inga.2019b,Lian.2022,Martirosyan.2024,Inga.2021,Lian.2024,Lin.2024,Wu.2024b}), the currently available offline methods \cite{Inga.2019b,Lian.2022,Martirosyan.2024} are derived for LQ DGs \cite{Inga.2019b,Martirosyan.2024} or insufficiently address the non-uniqueness of IDG solutions. The preferable goal would be the computation of a so-called solution set, i.e. the set of all equivalent cost function parameters that yield the observed GT trajectories. However, in \cite{Inga.2019b,Lian.2022,Martirosyan.2024}, only one element of this set is calculated. In addition, none of the mentioned offline methods provides an approximation error analysis. Multi-player HJB-based IDG methods that learn the cost functions of the players online \cite{Inga.2021,Lian.2024,Lin.2024,Wu.2024b} are all restricted to LQ DGs and again, lack a formal treatment of the non-uniqueness problem. 

Overall, there is no offline IDG method based on the coupled HJB equations that computes solution sets to analyze the non-uniqueness of IDG solutions beyond the scaling ambiguity. In addition, there is no online IDG method for nonlinear DGs\comment{based on the coupled HJB equations but also in general} with proofed convergence behavior and proofed connection of the converged parameters to the solution set of the offline approach. Both research gaps are closed in this paper by our new offline and online IDG methods presented in Sections~\ref{sec:offlineIDGHJB} and \ref{sec:onlineIDGHJB}, respectively. Moreover, for our new offline and online algorithms, we provide the first separate analysis of erroneous value and cost function approximations.

\section{Problem Formulation} \label{sec:problem}
Let the nonlinear, input-affine and time-invariant system dynamics be given by
\begin{align} \label{eq:x_dot}
	\dot{\bm{x}} = \bm{f}(\bm{x}) + \sum_{i=1}^{N} \bm{G}_i(\bm{x}) \bm{u}_i,
\end{align}
where $\bm{x} \in \mathbb{R}^n$ denotes the system state and $\bm{u}_i \in \mathbb{R}^{p_i}$ the control variable of player $i \in \mathcal{N} = \{1, \dots, N\}$, $N\in \mathbb{N}_{\geq 1}$. The system functions $\bm{f}: \mathbb{R}^n \rightarrow \mathbb{R}^n$ (with $\bm{f}(\bm{0})=\bm{0}$) and $\bm{G}_i: \mathbb{R}^n \rightarrow \mathbb{R}^{n \times p_i}$ are continuously differentiable on a neighborhood $\mathcal{X} \subset \mathbb{R}^n$ of the origin. Each player $i \in \mathcal{N}$ chooses an admissible feedback strategy $\bm{u}_i = \bm{\mu}_i(\bm{x}) \in \Gamma_i(\mathcal{X})$, where admissibility (see e.g. \cite[Definition~1]{Vamvoudakis.2011}) implies that $\bm{\mu}_i: \mathbb{R}^n \rightarrow \mathbb{R}^{p_i}$ (with $\bm{\mu}_i(\bm{0})=\bm{0}$) is continuously differentiable on $\mathcal{X}$ and that $\bm{\mu} = \mat{\bm{\mu}_1^\top & \dots & \bm{\mu}_N^\top}^\top$ ($\bm{\mu}: \mathbb{R}^n \rightarrow \mathbb{R}^p$, $p = \sum_{i=1}^{N} p_i$) stabilizes \eqref{eq:x_dot} on $\mathcal{X}$ and leads $\forall t\geq t_0$ and all possible initial states $\bm{x}(t_0) \in \mathcal{X}$ to a finite value function
\begin{align} \label{eq:V_i}
	V^{\bm{\mu}}_i(\bm{x}) = \int_{t}^{\infty} Q^*_i(\bm{x}) + \sum_{j=1}^{N} \bm{\mu}_j^\top \bm{R}^*_{ij} \bm{\mu}_j \text{d}\tau,
\end{align}
which is twice continuously differentiable on $\mathcal{X}$. In \eqref{eq:V_i}, $\bm{R}^*_{ij} \succeq \bm{0}, \forall i,j \in \mathcal{N}, i \neq j$, $\bm{R}^*_{ii} \succ \bm{0}$ and $Q^*_i: \mathbb{R}^n \rightarrow \mathbb{R}$ is a positive definite \cite[p.~53]{Narendra.2005}, continuously differentiable function. With \eqref{eq:x_dot}, \eqref{eq:V_i} and the strategy sets $\Gamma_i$ an $N$-player nonzero-sum DG is defined. We consider a noncooperative game setting, i.e. each player $i$ aims at choosing its strategy $\bm{\mu}_i$ such that \eqref{eq:V_i} is minimized without binding agreements with the other players. Thus, we assume that there exists a unique\comment{in general, uniqueness not necessary for trustworthiness of residual error (see CCTA paper); if not unique, trustworthiness, i.e. same FNE strategy as measured one when solving DG with identified parameters, depends on initial FNE guess; however, one suitable FNE strategy should exist}\comment{unique = unique set of value functions and thus, unique set of feedback strategies} feedback Nash equilibrium $\{\bm{\mu}^*_i \in \Gamma_i(\mathcal{X}): i \in \mathcal{N}\}$ according to Definition~\ref{def:FNE} the players seek. 
\begin{definition}[cf. {\cite[Definition~3.12]{Basar.1999}}] \label{def:FNE}
	An $N$-tuple of feedback strategies $\{\bm{\mu}^*_i \in \Gamma_i(\mathcal{X}): i \in \mathcal{N}\}$ is called feedback Nash equilibrium (FNE), if $V_i^* = V_i^{\bm{\mu}^*} \leq V_i^{\bm{\mu}^*_{-i}}, \forall i \in \mathcal{N}$ where $\bm{\mu}^* = \mat{{\bm{\mu}_1^*}^\top & \dots & {\bm{\mu}_N^*}^\top}^\top$ and $\bm{\mu}^*_{-i} = \mat{{\bm{\mu}_1^*}^\top & \dots & {\bm{\mu}_i}^\top & \dots & {\bm{\mu}_N^*}^\top}^\top$.
\end{definition}
With the made assumptions on \eqref{eq:x_dot}, \eqref{eq:V_i} and $\Gamma_i$, a necessary and sufficient condition (see \cite[Theorem~6.16]{Basar.1999} for the reduction to the single-player case under the FNE solution concept and \cite[p.~18]{Anderson.1989} for the single-player, i.e. OC, conditions) for the FNE is given by the coupled HJB equations 
\begin{align} \label{eq:hjb}
	0 &= \pdv{V_i^*}{\bm{x}} \left( \bm{f} + \sum_{j=1}^{N} \bm{G}_j \bm{\mu}_j^* \right) + \sum_{j=1}^{N} {\bm{\mu}_j^*}^\top \bm{R}^*_{ij} {\bm{\mu}_j^*} + Q^*_i(\bm{x}),
\end{align}
$\forall i \in \mathcal{N}$, where
\begin{align} \label{eq:mu_i}
	\bm{\mu}_i^*(\bm{x}) = - \frac{1}{2} {\bm{R}^*_{ii}}^{-1} {\bm{G}_i(\bm{x})}^\top \left( \pdv{V_i^*}{\bm{x}} \right)^\top.
\end{align}
The so-called forward problem, i.e. computing the FNE strategies $\{\bm{\mu}^*_i \in \Gamma_i(\mathcal{X}): i \in \mathcal{N}\}$ for given cost \eqref{eq:V_i} and system functions \eqref{eq:x_dot}, is typically solved by a policy iteration (PI) algorithm (see e.g. \cite[Algorithm~1]{Vamvoudakis.2011} or \cite{Karg.2023}). Here, the value functions $V_i^{\bm{\mu}}, \forall i \in \mathcal{N}$ corresponding to a strategy combination $\bm{\mu}$ are calculated in a policy evaluation step by solving \eqref{eq:hjb} with $\bm{\mu}$ for $V_i^{\bm{\mu}}, \forall i \in \mathcal{N}$. Then, a policy improvement with the achieved value functions follows according to \eqref{eq:mu_i}. This procedure is repeated iteratively until convergence to the FNE, i.e. $V_i^{\bm{\mu}} \rightarrow V_i^*, \forall i \in \mathcal{N}$ and $\bm{\mu} \rightarrow \bm{\mu}^*$ (see e.g. \cite[Theorem~1]{Liu.2014}).
\comment{PI algorithm constructs a unique sequence of value functions (solutions of Lyapunov equations) that converges to the value functions of the FNE; since we assume a unique FNE, we have convergence to this unique FNE (implicitly also assumed by typical ADP approaches); furthermore, if the basis functions for the value functions are chosen appropriately, one can show that for the sequence of Lyapunov equations a sequence of unique LS solutions for these parameterized value functions exists and thus, we have a unique sequence of value functions to the value functions of the unique FNE and a unique sequence of corresponding parameters of the parameterized value function structures to the parameters of these structures corresponding to the unique FNE}

As motivated in the introduction, we aim at solving the inverse problem of the DG defined by \eqref{eq:x_dot}, \eqref{eq:V_i} and $\Gamma_i$ offline and online as defined by Problem~\ref{problem:offline_idg} and \ref{problem:online_idg}.
\begin{assumption} \label{assm:Q_i}
	The state-dependent cost functions $Q^*_i(\bm{x})$, $\forall i \in \mathcal{N}$ are given by $Q^*_i(\bm{x}) = {\bm{\beta}^*_i}^{\top} \bm{\psi}_i(\bm{x})$, where $\bm{\psi}_i: \mathbb{R}^n \rightarrow \mathbb{R}^{m_i}$ is continuously differentiable on $\mathcal{X}$ and $\bm{\beta}^*_i \in \mathbb{R}^{m_i}$.
\end{assumption}

\begin{assumption} \label{assm:R_ij}
	The matrices $\bm{R}^*_{ij}$, $\forall i,j \in \mathcal{N}$ are diagonal.
\end{assumption}

\begin{assumption} \label{assm:V_i_star}
	The value functions $V_i^*(\bm{x})$, $\forall i \in \mathcal{N}$ corresponding to the FNE $\{\bm{\mu}^*_i \in \Gamma_i(\mathcal{X}): i \in \mathcal{N}\}$ are given by $V_i^*(\bm{x}) = {\bm{\theta}_i^*}^{\top} \bm{\phi}_i(\bm{x})$, where $\bm{\phi}_i: \mathbb{R}^n \rightarrow \mathbb{R}^{h_i}$ is twice continuously differentiable on $\mathcal{X}$ and $\bm{\theta}_i^* \in \mathbb{R}^{h_i}$.
\end{assumption}
\comment{these are the "parameterized cost and value function structures" defined in abstract and introduction}
\comment{this does not assume that all LE solutions can be approximated exactly with this linear combination (only relevant for the forward solution in the verification step); however, here a "sufficient" approximation is enough to ensure convergence of the PI algorithm and admissibility of the intermediate FNEs (implementation of PI algorithm needs to detect convergence to a residual set or implementation via grid-based LE evaluation and QP is used)}

\begin{assumption} \label{assm:systemDynamics_basisFunctions_known}
	The system dynamics functions $\bm{f}(\bm{x})$ and $\bm{G}_i(\bm{x}), \forall i \in \mathcal{N}$ as well as the basis functions $\bm{\psi}_i(\bm{x}), \forall i \in \mathcal{N}$ and $\bm{\phi}_i(\bm{x}), \forall i \in \mathcal{N}$ are known.
\end{assumption}

\begin{problem}[Offline Inverse Differential Game Problem] \label{problem:offline_idg}
	Let Assumptions~\ref{assm:Q_i}, \ref{assm:R_ij}, \ref{assm:V_i_star} and \ref{assm:systemDynamics_basisFunctions_known} hold. Furthermore, let $D \in \mathbb{N}_{\geq 1}$ demonstrations, i.e. trajectories ${\bm{x}^*}^{(d)}(t_0 \rightarrow \infty)$ and ${\bm{u}^*}^{(d)}(t_0 \rightarrow \infty)$ (GT trajectories), of the FNE $\{\bm{\mu}^*_i \in \Gamma_i(\mathcal{X}): i \in \mathcal{N}\}$ be given. Find (at least) one combination of parameters $\hat{\bm{\beta}}_i$ and $\hat{\bm{R}}_{ij}$, $\forall i,j \in \mathcal{N}$ such that ${\bm{x}^*}^{(d)}(t_0 \rightarrow \infty) = {\hat{\bm{x}}}^{*(d)}(t_0 \rightarrow \infty)$ and ${\bm{u}^*}^{(d)}(t_0 \rightarrow \infty) = {\hat{\bm{u}}}^{*(d)}(t_0 \rightarrow \infty)$, $\forall d \in \{1,\dots,D\}$, where ${\hat{\bm{x}}}^{*(d)}(t_0 \rightarrow \infty)$ and ${\hat{\bm{u}}}^{*(d)}(t_0 \rightarrow \infty)$ are demonstrations of the FNE $\{\hat{\bm{\mu}}_i^* \in \Gamma_i(\mathcal{X}): i \in \mathcal{N}\}$ resulting from $\hat{\bm{\beta}}_i$ and $\hat{\bm{R}}_{ij}$, $\forall i,j \in \mathcal{N}$.
\end{problem}

\begin{problem}[Online Inverse Differential Game Problem] \label{problem:online_idg}
	Let Assumptions~\ref{assm:Q_i}, \ref{assm:R_ij}, \ref{assm:V_i_star} and \ref{assm:systemDynamics_basisFunctions_known} hold. Furthermore, let the $N$ players act in the FNE $\{\bm{\mu}^*_i \in \Gamma_i(\mathcal{X}): i \in \mathcal{N}\}$ and let the corresponding trajectories $\bm{x}^*(t_0 \rightarrow t)$ and $\bm{u}^*(t_0 \rightarrow t)$ (GT trajectories) be given up to the current time $t$. Find (at least) one combination of parameters $\hat{\bm{\beta}}_i$ and $\hat{\bm{R}}_{ij}$, $\forall i,j \in \mathcal{N}$ such that $\bm{x}^*(t_0 \rightarrow \infty) = \hat{\bm{x}}^*(t_0 \rightarrow \infty)$ and $\bm{u}^*(t_0 \rightarrow \infty) = \hat{\bm{u}}^*(t_0 \rightarrow \infty)$, where $\hat{\bm{x}}^*(t_0 \rightarrow \infty)$ and $\hat{\bm{u}}^*(t_0 \rightarrow \infty)$ correspond to the FNE $\{\hat{\bm{\mu}}_i^* \in \Gamma_i(\mathcal{X}): i \in \mathcal{N}\}$ resulting from $\hat{\bm{\beta}}_i$ and $\hat{\bm{R}}_{ij}$, $\forall i,j \in \mathcal{N}$.
\end{problem}


\section{Offline Inverse Differential Game Method Based on Hamilton-Jacobi-Bellman Equations} \label{sec:offlineIDGHJB}
In this section, we propose a new offline IDG method based on the HJB equations~\eqref{eq:hjb} to solve Problem~\ref{problem:offline_idg}. The method is based on two main steps. Firstly, the $D$ demonstrated trajectories ${\bm{x}^*}^{(d)}(t_0 \rightarrow \infty)$ and ${\bm{u}^*}^{(d)}(t_0 \rightarrow \infty)$ are used to identify the FNE strategy \eqref{eq:mu_i} of each player $i$ via a quadratic program (QP) for each player. Through the FNE identification, we can also determine (at least parts of) the value function weights $\bm{\theta}_i^*$. In the second step, the identified FNE strategies $\hat{\bm{\mu}}_i(\bm{x})$ and the identified parts of $V_i^*$ are inserted into the HJB equations~\eqref{eq:hjb}. Since via this procedure the coupled HJB equations are reformulated as $N$ decoupled equations that need to hold for every $\bm{x} \in \mathcal{X}$, we can estimate $\bm{R}^*_{ij}$, $\bm{\beta}^*_i$ and the remaining parts of $\bm{\theta}_i^*$ by evaluating the $N$ decoupled HJB equations at arbitrarily chosen system states and constructing a second QP for each player. 

If Assumptions~\ref{assm:Q_i} and \ref{assm:V_i_star} are fulfilled, i.e. cost and value functions are given by known parameterized structures, Problem~\ref{problem:offline_idg} is guaranteed to be solved. The theoretical details of this case are discussed in Subsection~\ref{subsec:IDGHJB_offline_known_structures}. If the structures of cost and value functions are completely unknown and approximation structures, like polynomial basis functions, are used instead, Assumptions~\ref{assm:Q_i} and \ref{assm:V_i_star} are not fulfilled since approximation errors remain. Hence, in Subsection~\ref{subsec:IDGHJB_offline_approx_valueFct_structures}, we discuss the case of erroneous value function approximations (not fulfilled Assumption~\ref{assm:V_i_star}) which leads to bounded errors between the identified $\hat{\bm{\mu}}_i$ (and $\hat{\bm{\mu}}_i^*$ resulting from the IDG solution) and the GT FNE strategy $\bm{\mu}^*_i$. In Subsection~\ref{subsec:IDGHJB_online_approx_costFct_structures}, approximation errors in the second step of the IDG method are analyzed, i.e. Assumption~\ref{assm:Q_i} is not fulfilled.

\subsection{Known Cost and Value Function Structures} \label{subsec:IDGHJB_offline_known_structures}
In order to identify parts of the the value functions weights $\bm{\theta}_i^*$ alongside the FNE strategies in the first step of our IDG method, we firstly rearrange the basis functions $\bm{\phi}_i(\bm{x})$ according to ${\bm{\phi}_i(\bm{x})}^\top = \mat{ {\bm{\phi}^{(r)}_i(\bm{x})}^\top & {\bm{\phi}^{(-r)}_i(\bm{x})}^\top }$ and ${\bm{\theta}^*_i}^\top = \mat{ {\bm{\theta}^{(r)}_i}^{*\top} & {\bm{\theta}^{(-r)}_i}^{*\top} }$ (${\bm{\theta}^{(r)}_i} \in \mathbb{R}^{\bar{h}_i}$, ${\bm{\theta}^{(-r)}_i} \in \mathbb{R}^{h_i-\bar{h}_i}$) such that $\exists \bm{x} \in \mathcal{X}$ with
\begin{align} \label{eq:non_vanishing_basisFunctions}
	&{\bm{G}_i(\bm{x})}^\top \left( \pdv{\bm{\phi}^{(r)}_{i,j}(\bm{x})}{\bm{x}} \right)^\top \neq \bm{0}, \forall j \in \{1,\dots, \bar{h}_i\}
\end{align}
and such that $\forall \bm{x} \in \mathcal{X}$
\begin{align} \label{eq:vanishing_basisFunctions}
	{\bm{G}_i(\bm{x})}^\top \left( \pdv{\bm{\phi}^{(-r)}_{i,j}(\bm{x})}{\bm{x}} \right)^\top = \bm{0}, \forall j \in \{1,\dots, h_i-\bar{h}_i\}.
\end{align}

With this consideration and Assumption~\ref{assm:V_i_star}, we get 
\begin{align}
	\bm{\mu}_i^*(\bm{x}) &= - \frac{1}{2} {\bm{R}^*_{ii}}^{-1} \underbrace{{\bm{G}_i(\bm{x})}^\top \left( \pdv{\bm{\phi}^{(r)}_i(\bm{x})}{\bm{x}} \right)^\top}_{=\bm{\Phi}_{i}^{(r)}(\bm{x})} {\bm{\theta}^{(r)}_{i}}^* \notag \\
	&\hphantom{=} - \frac{1}{2} {\bm{R}^*_{ii}}^{-1} \underbrace{{\bm{G}_i(\bm{x})}^\top \left( \pdv{\bm{\phi}^{(-r)}_i(\bm{x})}{\bm{x}} \right)^\top}_{=\bm{0}} {\bm{\theta}^{(-r)}_i}^* \label{eq:mu_i_red}
\end{align}
from \eqref{eq:mu_i} and conclude that the parameters ${\bm{\theta}_i^{(-r)}}^*$ cannot be identified from the FNE strategies. In \eqref{eq:mu_i_red}, $\bm{\Phi}^{(r)}_i: \mathbb{R}^n \rightarrow \mathbb{R}^{p_i \times \bar{h}_i}$ is continuously differentiable and ${\bm{\Phi}^{(r)}_i(\bm{x})}^\top = \mat{\bar{\bm{\phi}}^{(r)}_{i,1}(\bm{x}) & \dots & \bar{\bm{\phi}}^{(r)}_{i,p_i}(\bm{x})}$. Now, Lemma~\ref{lemma:linear_combination_mu_i} states the final reformulation of the FNE strategy~\eqref{eq:mu_i}.
\begin{lemma} \label{lemma:linear_combination_mu_i}
	If Assumptions~\ref{assm:R_ij} and \ref{assm:V_i_star} hold, \eqref{eq:mu_i} can be written as 
	\begin{align} \label{eq:mu_i_linear}
		\bm{\mu}_i^*(\bm{x}) = -\frac{1}{2} \bar{\bm{\Phi}}_i(\bm{x}) \bar{\bm{\theta}}_i^*,
	\end{align}
	with
	\begin{align}
		\bar{\bm{\Phi}}_i(\bm{x}) &= \mat{\bm{0}_{1 \times 0} & {\bar{\bm{\phi}}_{i,1}^{(r)}(\bm{x})}^{\top} & \bm{0}_{1 \times (p_i-1)\bar{h}_i} \\ \bm{0}_{1 \times \bar{h}_i} & {\bar{\bm{\phi}}_{i,2}^{(r)}(\bm{x})}^{\top} & \bm{0}_{1 \times (p_i-2)\bar{h}_i} \\ \vdots & \vdots & \vdots \\ \bm{0}_{1 \times (p_i-1)\bar{h}_i} & {\bar{\bm{\phi}}_{i,p_i}^{(r)}(\bm{x})}^{\top} & \bm{0}_{1 \times 0}}, \label{eq:Phi_i_bar} \\
		\bar{\bm{\theta}}_i^* &= \mat{ {r^{*}_{ii,1}}^{-1}{\bm{\theta}_i^{(r)}}^* \\ {r^{*}_{ii,2}}^{-1}{\bm{\theta}_i^{(r)}}^* \\ \vdots \\ {r^{*}_{ii,p_i}}^{-1}{\bm{\theta}_i^{(r)}}^* }, \label{eq:theta_i_bar}
	\end{align}
	where $\bm{R}^*_{ii} = \diag\left(r^*_{ii,1}, r^*_{ii,2}, \dots, r^*_{ii,p_i}\right)$.
\end{lemma}
\begin{proof}
	By rearranging $\bm{\phi}_i(\bm{x})$ according to \eqref{eq:non_vanishing_basisFunctions} and \eqref{eq:vanishing_basisFunctions}, \eqref{eq:mu_i_red} results under Assumption~\ref{assm:V_i_star} from \eqref{eq:mu_i}. 
	With $\bm{R}^*_{ii} = \diag\left(r^*_{ii,1}, r^*_{ii,2}, \dots, r^*_{ii,p_i}\right)$ (Assumption~\ref{assm:R_ij}), \eqref{eq:mu_i_red} simplifies to 
	\begin{align} \label{eq:proof_mu_i_linear_1}
		\bm{\mu}_i^*(\bm{x}) = -\frac{1}{2} \mat{ {r^{*}_{ii,1}}^{-1} {\bar{\bm{\phi}}_{i,1}^{(r)}(\bm{x})}^{\top} {\bm{\theta}_i^{(r)}}^* \\ {r^{*}_{ii,2}}^{-1} {\bar{\bm{\phi}}_{i,2}^{(r)}(\bm{x})}^{\top} {\bm{\theta}_i^{(r)}}^* \\ \vdots \\ {r^{*}_{ii,p_i}}^{-1} {\bar{\bm{\phi}}_{i,p_i}^{(r)}(\bm{x})}^{\top} {\bm{\theta}_i^{(r)}}^* }.
	\end{align} 
	Finally, \eqref{eq:mu_i_linear} follows from \eqref{eq:proof_mu_i_linear_1} by defining $\bar{\bm{\theta}}_i^*$~\eqref{eq:theta_i_bar}.
\end{proof}

Lemma~\ref{lemma:linear_combination_mu_i} enables the identification of the FNE strategies $\bm{\mu}_i^*(\bm{x})$ and the corresponding value functions weights ${\bm{\theta}_i^{(r)}}^*$ from the $D$ demonstrations. Hereto, $K \in \mathbb{N}_{\geq 1}$ data points are chosen from each demonstration $d \in \{1, \dots, D\}$, ${\bm{x}^*}^{(d)}(t_0 \rightarrow \infty)$ and ${\bm{u}^*}^{(d)}(t_0 \rightarrow \infty)$, in equidistant time steps $\Delta t$:
\begin{align}
	\bm{M}_{u_i} &= \mat{ -\frac{1}{2}\bar{\bm{\Phi}}_i({\bm{x}^*}^{(1)}(t_0)) \\ -\frac{1}{2}\bar{\bm{\Phi}}_i({\bm{x}^*}^{(1)}(t_0 + \Delta t)) \\ \vdots \\ -\frac{1}{2}\bar{\bm{\Phi}}_i({\bm{x}^*}^{(1)}(t_0 + (K-1)\Delta t )) \\ -\frac{1}{2}\bar{\bm{\Phi}}_i({\bm{x}^*}^{(2)}(t_0)) \\ \vdots \\ -\frac{1}{2}\bar{\bm{\Phi}}_i({\bm{x}^*}^{(D)}(t_0 + (K-1) \Delta t)) }, \label{eq:M_u_full} \\
	\bm{z}_{u_i} &= \mat{ {\bm{u}_i^*}^{(1)}(t_0) \\ {\bm{u}_i^*}^{(1)}(t_0 + \Delta t) \\ \vdots \\ {\bm{u}_i^*}^{(1)}(t_0 + (K-1)\Delta t) \\ {\bm{u}_i^*}^{(2)}(t_0) \\ \vdots \\ {\bm{u}_i^*}^{(D)}(t_0 + (K-1)\Delta t) }. \label{eq:z_u_full}
\end{align}
The number of demonstrations $D$ and the initial states ${\bm{x}^*}^{(d)}(t_0)$ of them as well as the sampling parameters $K$ and $\Delta t$ should be chosen such that Assumption~\ref{assm:excitation_FNE_offline} for enough excitation in the observed data is fulfilled.
\begin{assumption} \label{assm:excitation_FNE_offline}
	The number of demonstrations $D$, the initial states ${\bm{x}^*}^{(d)}(t_0)$, $K$ and $\Delta t$ in \eqref{eq:M_u_full} and \eqref{eq:z_u_full} are chosen such that $\rank\left(\bm{M}_{u_i}\right)\leq p_i\bar{h}_i$ shows the highest possible column rank, i.e. further data points do not increase $\rank\left(\bm{M}_{u_i}\right)$.
\end{assumption}

Now, Lemma~\ref{lemma:FNE_ident_offline} yields the identification $\hat{\bm{\mu}}_i(\bm{x})$ of the FNE strategy $\bm{\mu}^*_i(\bm{x})$ of one player $i \in \mathcal{N}$ as well as an estimate $\hat{\bm{\theta}}_i^{(r)}$ of the reduced value function weight vector ${\bm{\theta}_i^{(r)}}^*$.
\begin{lemma} \label{lemma:FNE_ident_offline}
	If Assumptions~\ref{assm:R_ij}, \ref{assm:V_i_star}, \ref{assm:systemDynamics_basisFunctions_known} and \ref{assm:excitation_FNE_offline} hold,
	\begin{align} \label{eq:theta_bar_hat}
		\hat{\bar{\bm{\theta}}}_i = \left(\bm{M}_{u_i}^\top \bm{M}_{u_i}\right)^{+} \bm{M}_{u_i}^\top \bm{z}_{u_i}
	\end{align}
	yields
	\begin{align} \label{eq:mu_i_hat}
		\bm{\mu}_i^*(\bm{x}) = \hat{\bm{\mu}}_i(\bm{x}) = -\frac{1}{2} \bar{\bm{\Phi}}_i(\bm{x}) \hat{\bar{\bm{\theta}}}_i.
	\end{align}
	If additionally $\rank\left(\bm{M}_{u_i}\right)=p_i\bar{h}_i$, we can find $c_i \in \mathbb{R}_{> 0}$ such that with $\bm{P}_j = \mat{ \bm{0}_{\bar{h}_i \times (j-1)\bar{h}_i} & \bm{I}_{\bar{h}_i\times\bar{h}_i} & \bm{0}_{\bar{h}_i \times (p_i-j)\bar{h}_i} }$
	\begin{align} \label{eq:theta_r_hat}
		\hat{\bm{\theta}}_i^{(r)} = \frac{1}{p_i} \sum_{j=1}^{p_i} \frac{\bm{P}_j \hat{\bar{\bm{\theta}}}_i }{\norm{\bm{P}_j \hat{\bar{\bm{\theta}}}_i }_2} = c_i {\bm{\theta}_i^{(r)}}^*.
	\end{align}
\end{lemma}
\begin{proof}
	From Lemma~\ref{lemma:linear_combination_mu_i}, we derive
	\begin{align} \label{eq:proof_FNE_ident_offline_1}
		\bm{z}_{u_i} = \bm{M}_{u_i} \bar{\bm{\theta}}_i^*.
	\end{align}
	In order to determine $\bar{\bm{\theta}}_i^*$ and thus, identify $\bm{\mu}_i^*(\bm{x})$, we define $\bm{\epsilon}_{u_i} = \bm{M}_{u_i} \bar{\bm{\theta}}_i - \bm{z}_{u_i}$ and set up the convex QP
	\begin{align} \label{eq:proof_FNE_ident_offline_2}
		&\min_{\bar{\bm{\theta}}_i} \left\{ \frac{1}{2} \bm{\epsilon}_{u_i}^\top \bm{\epsilon}_{u_i} \right\}= \notag \\ 
		&\min_{\bar{\bm{\theta}}_i} \left\{ \frac{1}{2} \bar{\bm{\theta}}_i^\top \bm{M}_{u_i}^\top \bm{M}_{u_i} \bar{\bm{\theta}}_i - \bm{z}_{u_i}^\top \bm{M}_{u_i} \bar{\bm{\theta}}_i + \frac{1}{2} \bm{z}_{u_i}^\top\bm{z}_{u_i} \right\}.
	\end{align}
	Since \eqref{eq:proof_FNE_ident_offline_1} holds, all global minimizers 
	\begin{align} \label{eq:proof_FNE_ident_offline_3}
		\hat{\bar{\bm{\theta}}}_i = \bar{\bm{M}}_{u_i}^{+} \bm{M}_{u_i}^\top \bm{z}_{u_i} + \left( \bm{I} - \bar{\bm{M}}_{u_i}^{+} \bar{\bm{M}}_{u_i} \right) \bm{w}_i,
	\end{align}
	where $\bar{\bm{M}}_{u_i} = \left({\bm{M}_{u_i}}^\top \bm{M}_{u_i}\right)$, $(\cdot)^+$ denotes the Moore-Penrose inverse and $\bm{w}_i \in \mathbb{R}^{p_i\bar{h}_i}$ an arbitrary vector, of QP~\eqref{eq:proof_FNE_ident_offline_2} yield $\bm{\epsilon}_{u_i}=\bm{0}$. From Assumption~\ref{assm:excitation_FNE_offline}, we conclude that further data points do not shrink the set of global minimizers and that every $\hat{\bar{\bm{\theta}}}_i$~\eqref{eq:proof_FNE_ident_offline_3} yields \eqref{eq:mu_i_hat}, also the one with $\bm{w}_i=\bm{0}$, i.e. \eqref{eq:theta_bar_hat}. If $\rank\left(\bm{M}_{u_i}\right)=p_i\bar{h}_i$, the QP is strict convex and \eqref{eq:theta_bar_hat} is the unique minimizer with $\hat{\bar{\bm{\theta}}}_i=\bar{\bm{\theta}}_i^*$. Hence, we get with
	\begin{align} \label{eq:proof_FNE_ident_offline_4}
		\hat{\bm{\theta}}_i^{(r)} = \frac{1}{p_i} \sum_{j=1}^{p_i} \frac{\bm{P}_j \hat{\bar{\bm{\theta}}}_i }{\norm{\bm{P}_j \hat{\bar{\bm{\theta}}}_i }_2} = \frac{{\bm{\theta}_i^{(r)}}^*}{\norm{{\bm{\theta}_i^{(r)}}^*}_2} = c_i {\bm{\theta}_i^{(r)}}^*
	\end{align}
	relation~\eqref{eq:theta_r_hat} to the searched parameters~${\bm{\theta}_i^{(r)}}^*$.
\end{proof}

Now, in the second step of our IDG method, we use the identified FNE strategies $\hat{\bm{\mu}}_i(\bm{x})$ of each player $i$ and the estimate $\hat{\bm{\theta}}_i^{(r)}$ of one part of the value function weights in the HJB equations \eqref{eq:hjb} to solve Problem~\ref{problem:offline_idg}. If $\rank\left(\bm{M}_{u_i}\right)=p_i\bar{h}_i$, using Assumption~\ref{assm:Q_i}, \ref{assm:R_ij} and \ref{assm:V_i_star}, we reformulate the coupled HJB equations~\eqref{eq:hjb} as
\begin{align} 
	\underbrace{\mat{ \bm{\mu}^{*\top} \odot \bm{\mu}^{*\top} & \bm{\psi}_i^\top & \left( \pdv{\bm{\phi}_i^{(-r)}}{\bm{x}} \bm{f}^*_g \right)^\top }}_{{\bm{m}_{\hjb_i}^{(r)}(\bm{x})}^\top} \mat{ \bm{\alpha}^*_i \\ \bm{\beta}^*_i \\ {\bm{\theta}_i^{(-r)}}^* } = \notag \\ \underbrace{- {\bm{\theta}_i^{(r)}}^{*\top} \pdv{\bm{\phi}_i^{(r)}}{\bm{x}} \bm{f}^*_g}_{z^{(r)}_{\hjb_i}(\bm{x})}, \label{eq:hjb_linear_1}
\end{align}
where $\bm{f}^*_g = \bm{f} + \sum_{j=1}^{N} \bm{G}_j \bm{\mu}_j^*$, ${\bm{\alpha}^*_i}^\top = \mat{ r^*_{i1,1} & \dots & r^*_{i1,p_1} & r^*_{i2,1} & \dots & r^*_{iN,p_N} }$ (diagonal elements of $\bm{R}^*_{ij}$, $\forall j \in \mathcal{N}$) and $\odot$ is the Hadamard product. If $\rank\left(\bm{M}_{u_i}\right)<p_i\bar{h}_i$, we reformulate \eqref{eq:hjb} as
\begin{align} 
	\underbrace{\mat{ \bm{\mu}^{*\top} \odot \bm{\mu}^{*\top} & \bm{\psi}_i^\top & \left( \pdv{\bm{\phi}_i}{\bm{x}} \bm{f}^*_g \right)^\top }}_{{\bm{m}_{\hjb_i}(\bm{x})}^\top} \mat{ \bm{\alpha}^*_i \\ \bm{\beta}^*_i \\ {\bm{\theta}_i}^* } = \underbrace{0}_{z_{\hjb_i}(\bm{x})}. \label{eq:hjb_linear_2}
\end{align}

Since \eqref{eq:hjb_linear_1} and \eqref{eq:hjb_linear_2} hold $\forall \bm{x} \in \mathcal{X}$, we choose $\bar{K} \in \mathbb{N}_{\geq 1}$ arbitrary data points $\bm{x}_{\bar{k}} \in \mathcal{X}$ and define
\begin{align}
	\bm{M}^{(r)}_{\hjb_i} &= \mat{ {\bm{m}_{\hjb_i}^{(r)}(\bm{x}_1)}^\top \\ \vdots \\ {\bm{m}_{\hjb_i}^{(r)}(\bm{x}_{\bar{K}})}^\top }, \label{eq:M_hjb_full_red} \\
	\bm{z}^{(r)}_{\hjb_i} &= \mat{ z_{\hjb_i}^{(r)}(\bm{x}_1) \\ \vdots \\ z_{\hjb_i}^{(r)}(\bm{x}_{\bar{K}}) } \label{eq:z_hjb_full_red}
\end{align}
and $\bm{M}_{\hjb_i}$ and $\bm{z}_{\hjb_i}$ with $\bm{m}_{\hjb_i}(\bm{x})$ and $z_{\hjb_i}(\bm{x})$ accordingly. 
The points $\bm{x}_{\bar{k}} \in \mathcal{X}$ need to be chosen such that the excitation Assumption~\ref{assm:excitation_HJB_offline} is fulfilled.
\begin{assumption} \label{assm:excitation_HJB_offline}
	The data points $\bm{x}_{\bar{k}}$ to set up $\bm{M}^{(r)}_{\hjb_i}$, $\bm{z}^{(r)}_{\hjb_i}$ and $\bm{M}_{\hjb_i}$, $\bm{z}_{\hjb_i}$, respectively, are chosen such that $\rank \left( \bm{M}^{(r)}_{\hjb_i} \right)\leq p + m_i + h_i - \bar{h}_i$ and $\rank \left( \bm{M}_{\hjb_i} \right)\leq p + m_i + h_i$, respectively, shows the highest possible column rank, i.e. further data points do not increase $\rank \left( \bm{M}^{(r)}_{\hjb_i} \right)$ and $\rank \left( \bm{M}_{\hjb_i} \right)$, respectively.
\end{assumption}
\comment{data points should be sufficiently distributed in $\mathcal{X}$, e.g. equidistant grid around origin typically a good choice (see numerical example)}

We provide our first main result in Theorem~\ref{theorem:HJB_ident_offline}, i.e. the solution to Problem~\ref{problem:offline_idg}, for the two different cases occurring in the FNE identification of player $i \in \mathcal{N}$: 1) $\rank\left(\bm{M}_{u_i}\right)=p_i\bar{h}_i$ and 2) $\rank\left(\bm{M}_{u_i}\right)<p_i\bar{h}_i$.  
\begin{theorem} \label{theorem:HJB_ident_offline}
	Let Assumptions~\ref{assm:Q_i}, \ref{assm:R_ij}, \ref{assm:V_i_star}, \ref{assm:systemDynamics_basisFunctions_known}, \ref{assm:excitation_FNE_offline} and \ref{assm:excitation_HJB_offline} hold $\forall i \in \mathcal{N}$. Furthermore, let the identification of the FNE strategy $\bm{\mu}_i^*(\bm{x})$ $\forall i \in \mathcal{N}$ be performed by Lemma~\ref{lemma:FNE_ident_offline}. If for player $i \in \mathcal{N}$ 
	\begin{enumerate}
		\item $\rank\left(\bm{M}_{u_i}\right)=p_i\bar{h}_i$, we define the set of parameters
		\begin{align} \label{eq:HJB_ident_offline_solSet_1}
			\mat{ \hat{\bm{\alpha}}_i \\ \hat{\bm{\beta}}_i } &= \bm{L}^{(r)}_i \bar{\bm{M}}^{(r)^+}_{\hjb_i}\bm{M}^{(r)^\top}_{\hjb_i} \bm{z}^{(r)}_{\hjb_i} \notag \\
			&\hphantom{=}+ \bm{L}^{(r)}_i \left(\bm{I} - \bar{\bm{M}}^{(r)^+}_{\hjb_i} \bar{\bm{M}}^{(r)}_{\hjb_i} \right) \bar{\bm{w}}^{(r)}_i,
		\end{align}
		where $\bm{L}^{(r)}_i = \mat{ \bm{I}_{p+m_i \times p+m_i} & \bm{0}_{p+m_i \times h_i-\bar{h}_i} }$, $\bar{\bm{M}}^{(r)}_{\hjb_i} = \bm{M}^{(r)^\top}_{\hjb_i}\bm{M}^{(r)}_{\hjb_i}$ and $\bar{\bm{w}}^{(r)}_i \in \mathbb{R}^{p+m_i+h_i-\bar{h}_i}$ is arbitrary but such that $\hat{\bm{\alpha}}_i$ and $\hat{\bm{\beta}}_i$ guarantee that the decoupled HJB equation (\eqref{eq:hjb} with $\hat{\bm{\mu}}_i(\bm{x})$ for $\bm{\mu}^*_i(\bm{x}),\forall i \in \mathcal{N}$) is necessary and sufficient for a unique OC solution\comment{unique solution = unique value function and thus, unique control law; results in LQ case (ARE) existent, but scarce in nonlinear case, only results for unique LE solutions typically to ensure unique sequence of value functions to optimal value functions, which however only guarantees convergence of a PI algorithm but does not say anything about the uniqueness of the solution - through the assumption in problem formulation on existence of unique FNE, unique OC solutions here in principle possible}. The set~\eqref{eq:HJB_ident_offline_solSet_1} consists of only one element $\hat{\bm{\alpha}}_i = c_i \bm{\alpha}^*_i$ and $\hat{\bm{\beta}}_i = c_i \bm{\beta}^*_i$ if and only if $\bm{L}^{(r)}_i \left(\bm{I} - \bar{\bm{M}}^{(r)^+}_{\hjb_i} \bar{\bm{M}}^{(r)}_{\hjb_i} \right) = \bm{0}$.
		\item $\rank\left(\bm{M}_{u_i}\right)<p_i\bar{h}_i$, we define the set of parameters
		\begin{align} \label{eq:HJB_ident_offline_solSet_2}
			\mat{ \hat{\bm{\alpha}}_i \\ \hat{\bm{\beta}}_i } &= \bm{L}_i \bar{\bm{M}}^{+}_{\hjb_i}\bm{M}^{\top}_{\hjb_i} \bm{z}_{\hjb_i} \notag \\
			&\hphantom{=}+  \bm{L}_i \left(\bm{I} - \bar{\bm{M}}^{+}_{\hjb_i} \bar{\bm{M}}_{\hjb_i} \right) \bar{\bm{w}}_i,
		\end{align}
		where $\bm{L}_i = \mat{ \bm{I}_{p+m_i \times p+m_i} & \bm{0}_{p+m_i \times h_i} }$, $\bar{\bm{M}}_{\hjb_i} = \bm{M}^{\top}_{\hjb_i}\bm{M}_{\hjb_i}$ and $\bar{\bm{w}}_i \in \mathbb{R}^{p+m_i+h_i}$ is arbitrary but such that $\hat{\bm{\alpha}}_i$ and $\hat{\bm{\beta}}_i$ guarantee that the decoupled HJB equation (\eqref{eq:hjb} with $\hat{\bm{\mu}}_i(\bm{x})$ for $\bm{\mu}^*_i(\bm{x}),\forall i \in \mathcal{N}$) is necessary and sufficient for a unique OC solution\comment{unique solution = unique value function and thus, unique control law}. The set~\eqref{eq:HJB_ident_offline_solSet_2} consists of only one element $\hat{\bm{\alpha}}_i = \bm{\alpha}^*_i$ and $\hat{\bm{\beta}}_i = \bm{\beta}^*_i$ if and only if $\bm{L}_i \left(\bm{I} - \bar{\bm{M}}^{+}_{\hjb_i} \bar{\bm{M}}_{\hjb_i} \right) = \bm{0}$.
	\end{enumerate}
	All parameters $\hat{\bm{\alpha}}_i$, $\hat{\bm{\beta}}_i , \forall i \in \mathcal{N}$ according to the sets \eqref{eq:HJB_ident_offline_solSet_1} and \eqref{eq:HJB_ident_offline_solSet_2}, respectively, solve Problem~\ref{problem:offline_idg}.
\end{theorem}
\begin{proof}
	If for all players $i \in \mathcal{N}$ the FNE strategy identification is performed according to Lemma~\ref{lemma:FNE_ident_offline}, we have $\bm{\mu}^*_i(\bm{x}) = \hat{\bm{\mu}}_i(\bm{x}), \forall i \in \mathcal{N}$ and get with Assumption~\ref{assm:Q_i}, \ref{assm:R_ij} and \ref{assm:V_i_star} $N$ decoupled equations by \eqref{eq:hjb_linear_1} or \eqref{eq:hjb_linear_2} (depending on $\rank\left(\bm{M}_{u_i}\right)$ in the FNE identification) with $\bm{f}_g^* = \hat{\bm{f}}_g = \bm{f} + \sum_{j=1}^{N} \bm{G}_j \hat{\bm{\mu}}_j$ and $\bm{\mu}^* = \hat{\bm{\mu}}$. Eqs.~\eqref{eq:hjb_linear_1} and \eqref{eq:hjb_linear_2} only depend on parameters of player $i$. Now, if for every player $i$ the set of parameters $\hat{\bm{\alpha}}_i$, $\hat{\bm{\beta}}_i$ is computed that fulfills \eqref{eq:hjb_linear_1} or \eqref{eq:hjb_linear_2} and that guarantees that the decoupled HJB equation is necessary and sufficient for a unique OC solution, all parameter combinations $\hat{\bm{\alpha}}_i$, $\hat{\bm{\beta}}_i, \forall i \in \mathcal{N}$ yield a FNE $\{\hat{\bm{\mu}}^*_i: i \in \mathcal{N}\}$ with $\hat{\bm{\mu}}^*_i(\bm{x})=\hat{\bm{\mu}}_i(\bm{x})=\bm{\mu}^*_i(\bm{x}), \forall i \in \mathcal{N}$. Due to the assumed differentiability characteristics on the system dynamics functions (see Section~\ref{sec:problem}), the FNE $\{\hat{\bm{\mu}}^*_i: i \in \mathcal{N}\}$ yields unique trajectories ${\hat{\bm{x}}}^{*(d)}(t_0 \rightarrow \infty)$ and ${\hat{\bm{u}}}^{*(d)}(t_0 \rightarrow \infty), \forall d \in \{1,\dots,D\}$ with ${\hat{\bm{x}}}^{*(d)}(t_0 \rightarrow \infty) = {\bm{x}^*}^{(d)}(t_0 \rightarrow \infty)$ and ${\hat{\bm{u}}}^{*(d)}(t_0 \rightarrow \infty) = {\bm{u}^*}^{(d)}(t_0 \rightarrow \infty)$. 
	
	Now, we show that for every player $i$ the sets of parameters $\hat{\bm{\alpha}}_i$, $\hat{\bm{\beta}}_i$  are given by \eqref{eq:HJB_ident_offline_solSet_1} and \eqref{eq:HJB_ident_offline_solSet_2}. In case of $\rank\left(\bm{M}_{u_i}\right)=p_i\bar{h}_i$, we compute $\bm{M}_{\hjb_i}^{(r)}$ and $\bm{z}_{\hjb_i}^{(r)}$ with $\bm{f}_g^* = \hat{\bm{f}}_g = \bm{f} + \sum_{j=1}^{N} \bm{G}_j \hat{\bm{\mu}}_j$ and ${\bm{\theta}_i^{(r)}}^* = \hat{\bm{\theta}}_i^{(r)}$ such that 
	\begin{align} \label{eq:proof_HJB_ident_offline_1}
		\bm{M}^{(r)}_{\hjb_i} \bm{\eta}^{(r)^*}_i = \bm{z}^{(r)}_{\hjb_i}, \bm{\eta}^{(r)^*}_i= \mat{ c_i{\bm{\alpha}^*_i}^\top \!\!&\!\! c_i{\bm{\beta}^*_i}^\top \!\! & \!\! c_i{\bm{\theta}_i^{(-r)}}^{*\top} }^\top
	\end{align}
	follows from \eqref{eq:hjb_linear_1}. As for the FNE identification, we define $\bm{\epsilon}^{(r)}_{\hjb_i}=\bm{M}^{(r)}_{\hjb_i}\bm{\eta}^{(r)}_i-\bm{z}^{(r)}_{\hjb_i}$ and set up the convex QP
	\begin{align} \label{eq:proof_HJB_ident_offline_2}
		&\min_{\bm{\eta}_i^{(r)}} \left\{ \frac{1}{2} {\bm{\epsilon}^{(r)}_{\hjb_i}}^\top \bm{\epsilon}^{(r)}_{\hjb_i} \right\}, 
	\end{align}
	whose global minimizers yield $\bm{\epsilon}^{(r)}_{\hjb_i}=\bm{0}$ due to \eqref{eq:proof_HJB_ident_offline_1} and can be computed by \eqref{eq:HJB_ident_offline_solSet_1} without the multiplications with $\bm{L}_i^{(r)}$. If Assumption~\ref{assm:excitation_HJB_offline} holds, the smallest set of global minimizers are guaranteed and we find $\bar{\bm{w}}^{(r)}_i$ such that \eqref{eq:HJB_ident_offline_solSet_1} yields $\hat{\bm{\alpha}}_i = c_i \bm{\alpha}^*_i$ and $\hat{\bm{\beta}}_i = c_i \bm{\beta}^*_i$, which in turn yield the FNE $\{\hat{\bm{\mu}}^*_i: i \in \mathcal{N}\}$ with $\hat{\bm{\mu}}^*_i(\bm{x})=\hat{\bm{\mu}}_i(\bm{x})=\bm{\mu}^*_i(\bm{x}), \forall i \in \mathcal{N}$ (see assumption in Section~\ref{sec:problem} that HJB equations are necessary and sufficient for a unique FNE with $\bm{\alpha}^*_i$, $\bm{\beta}^*_i$\comment{here, an appropriate scaling $c_i$ is assumed such that $\hat{\bm{\alpha}}_i$, $\hat{\bm{\beta}}_i$ still shape the HJB equations necessary and sufficient for a unique FNE}). Now, if arbitrary vectors $\bar{\bm{w}}^{(r)}_i$ are chosen such that the parameters~\eqref{eq:HJB_ident_offline_solSet_1} still guarantee that the decoupled HJB equation~\eqref{eq:hjb_linear_1} is necessary and sufficient for a unique OC solution, they all yield $\{\hat{\bm{\mu}}^*_i: i \in \mathcal{N}\}$ with $\hat{\bm{\mu}}^*_i(\bm{x})=\hat{\bm{\mu}}_i(\bm{x})=\bm{\mu}^*_i(\bm{x}), \forall i \in \mathcal{N}$. Obviously, if and only if $\bm{L}^{(r)}_i \left(\bm{I} - \bar{\bm{M}}^{(r)^+}_{\hjb_i} \bar{\bm{M}}^{(r)}_{\hjb_i} \right) = \bm{0}$, we compute unique parameters by \eqref{eq:HJB_ident_offline_solSet_1}.
	In case of $\rank\left(\bm{M}_{u_i}\right)<p_i\bar{h}_i$, we follow the same logic to derive the set~\eqref{eq:HJB_ident_offline_solSet_2} but define the QP based on $\bm{\epsilon}_{\hjb_i} = \bm{M}_{\hjb_i} \bm{\eta}_i - \bm{z}_{\hjb_i}$ and its guaranteed global minimizer $\bm{\eta}^{*\top}_i = \mat{ {\bm{\alpha}^*_i}^\top \!\!&\!\! {\bm{\beta}^*_i}^\top \!\! & \!\! {\bm{\theta}_i^*}^{\top} }$.
\end{proof}

For reproducibility, Algorithm~\ref{alg:offlineIDGHJB} summarizes the procedure of our new offline IDG method.


\begin{remark} \label{remark:benefit_VFweightsIdent}
	Using the identified value function weights~$\hat{\bm{\theta}}_i^{(r)}$ in the HJB identification (when $\rank\left(\bm{M}_{u_i}\right)=p_i\bar{h}_i$) provides two benefits: Assumption~\ref{assm:excitation_HJB_offline} is easier to fulfill and the non-relevant solution $\hat{\bm{\eta}}_i=\bm{0}$ is avoided by default.
\end{remark}
\comment{in general, reformulating the FNE strategies based on the basis functions is necessary to enable fulfillment of Assumption~\ref{assm:excitation_FNEident_online} and thus, functioning of the online method}

\begin{algorithm}[t]
	\caption{Offline HJB-Based IDG Method}
	\label{alg:offlineIDGHJB}
	\DontPrintSemicolon
	
	\KwIn{GT trajectories ${\bm{x}^*}^{(d)}(t_0 \rightarrow \infty)$ and ${\bm{u}^*}^{(d)}(t_0 \rightarrow \infty)$ ($d \in \{1,\dots,D\}$), system dynamics functions $\bm{f}(\bm{x})$ and $\bm{G}_i(\bm{x}), \forall i \in \mathcal{N}$ and basis functions $\bm{\psi}_i(\bm{x})$ and $\bm{\phi}_i(\bm{x}), \forall i \in \mathcal{N}$.}
	\KwOut{Cost function parameters $\hat{\bm{\beta}}_i$ and $\hat{\bm{R}}_{ij}$, $\forall i,j \in \mathcal{N}$.}
	
	\For{$i \in \mathcal{N}$}{
		Set up data matrices $\bm{M}_{u_i}$~\eqref{eq:M_u_full} and $\bm{z}_{u_i}$~\eqref{eq:z_u_full}. \;
		Compute $\hat{\bm{\mu}}_i(\bm{x})$ according to \eqref{eq:mu_i_hat} with \eqref{eq:theta_bar_hat}. \;
		\If{$\rank\left(\bm{M}_{u_i}\right)=p_i\bar{h}_i$}{
			Compute $\hat{\bm{\theta}}_i^{(r)}$ with \eqref{eq:theta_r_hat}.
		}
	}

	\For{$i \in \mathcal{N}$}{
		\If{$\rank\left(\bm{M}_{u_i}\right)=p_i\bar{h}_i$}{
			Set up data matrices $\bm{M}_{\hjb_i}^{(r)}$~\eqref{eq:M_hjb_full_red} and $\bm{z}_{\hjb_i}^{(r)}$~\eqref{eq:z_hjb_full_red} from \eqref{eq:hjb_linear_1}, where ${\bm{\theta}_i^{(r)}}^*=\hat{\bm{\theta}}_i^{(r)}$ and $\bm{\mu}_i^*(\bm{x})=\hat{\bm{\mu}}_i(\bm{x}), \forall i \in \mathcal{N}$. \;
			\Return $\hat{\bm{\beta}}_i$ and $\hat{\bm{R}}_{ij}, \forall j \in \mathcal{N}$ according to \eqref{eq:HJB_ident_offline_solSet_1}.
		}
		\Else{
			Set up data matrices $\bm{M}_{\hjb_i}$ and $\bm{z}_{\hjb_i}$ from \eqref{eq:hjb_linear_2}, where $\bm{\mu}_i^*(\bm{x})=\hat{\bm{\mu}}_i(\bm{x}), \forall i \in \mathcal{N}$. \;
			\Return $\hat{\bm{\beta}}_i$ and $\hat{\bm{R}}_{ij}, \forall j \in \mathcal{N}$ according to \eqref{eq:HJB_ident_offline_solSet_2}.
		}
	}
\end{algorithm}


\subsection{Approximated Value Function Structures} \label{subsec:IDGHJB_offline_approx_valueFct_structures}
\tdd{rework, shorten: only give formulas for errors, state boundedness of these errors (according to Stone-Weierstrass theorem), discuss influence cost function structure approximation errors and value function structure approximation errors (lemma) separately, discuss that cost function approximation errors are in line with results CCTA paper, remark that under certain conditions overfitting FNE identification good idea} 


If the value functions are not given by known parameterized structures (see Assumption~\ref{assm:V_i_star}), approximation structures can be used in the FNE identification step. For our analysis, we assume that the basis functions~$\bm{\phi}_i, \forall i \in \mathcal{N}$ are chosen based on the Stone-Weierstrass theorem~\cite[Theorem~3, p.~24]{Mhaskar.2000}. However, the results are transferable to other approximation theorems as long as the differentiability characteristics of the GT value functions and the basis functions from Section~\ref{sec:problem} are ensured. Lemma~\ref{lemma:FNEident_valueFct_error_offline} states that a bounded error between unknown GT FNE strategy and the identified strategy results\comment{trajectory error results via the system dynamics equations, hard to derive analytical statement, but it is bounded}.
\begin{lemma} \label{lemma:FNEident_valueFct_error_offline}
	Let Assumptions~\ref{assm:R_ij}, \ref{assm:systemDynamics_basisFunctions_known} and \ref{assm:excitation_FNE_offline} hold. If Assumption~\ref{assm:V_i_star} is not fulfilled but the basis functions~$\bm{\phi}_i$ are from a subalgebra (cf.~\cite[pp.~22f.]{Mhaskar.2000}) of the set of bounded and twice continuously differentiable functions on $\mathcal{X}$\comment{implicitly assumed that with appropriate basis functions/subalgebra the theorem extends to the approximation of bounded and twice continuously differentiable functions on $\mathcal{X}$; normally, the theorem only holds for continuous functions; if the function is however also continuously differentiable and the basis functions (e.g. polynomials) that fulfill the conditions of the approximation theorem as well; then, we have the differentiability and boundedness characteristics for the approximated function, the basis functions and the error} which fulfill the conditions in \cite[Theorem~3, p.~24]{Mhaskar.2000}\comment{e.g. polynomials, more general approximation theorems as the Stone-Weierstrass theorem for NNs allow also non-differentiable/non-continuous basis/activation functions, but then implicit extension to twice continuously differentiable functions problematic and the characteristics on error function cannot be guaranteed}, a bounded, twice continuously differentiable function $\bar{\varepsilon}_i: \mathcal{X} \rightarrow \mathbb{R}$ exists such that $\forall \bm{x} \in \mathcal{X}$ $V_i^*(\bm{x}) = {\bm{\theta}_i^*}^\top \bm{\phi}_i(\bm{x}) + \bar{\varepsilon}_i(\bm{x})$. Now, let a control law $\tilde{\bm{\mu}}_i^*(\bm{x}) = -\frac{1}{2} {\bm{R}_{ii}^*}^{-1} \bm{G}_i^\top \left(\pdv{\bm{\phi}_i}{\bm{x}}\right)^\top \bm{\theta}_i^*$ be defined. Then, $\tilde{\bm{\mu}}_i^*(\bm{x})=-\frac{1}{2} \bar{\bm{\Phi}}_i(\bm{x}) \bar{\bm{\theta}}_i^*$ and $\hat{\bar{\bm{\theta}}}_i$~\eqref{eq:theta_bar_hat} yields $\tilde{\bm{\mu}}_i^*(\bm{x}) = \hat{\bm{\mu}}_i(\bm{x}) = -\frac{1}{2} \bar{\bm{\Phi}}_i(\bm{x}) \hat{\bar{\bm{\theta}}}_i$ with
	\begin{align} \label{eq:bounded_FNE_error}
		\norm{\hat{\bm{\mu}}_i(\bm{x}) - \bm{\mu}_i^*(\bm{x}) } = \norm{ \frac{1}{2} \bm{R}_{ii}^{*-1} \bm{G}^\top_i \pdv{\bar{\varepsilon}_i}{\bm{x}}^\top } < \infty.
	\end{align}
	Moreover, if additionally $\rank\left(\bm{M}_{u_i}\right)=p_i\bar{h}_i$, \eqref{eq:theta_r_hat} holds.
\end{lemma}
\begin{proof}
	Since the GT FNE $\{\bm{\mu}^*_i: i \in \mathcal{N}\}$ stabilizes \eqref{eq:x_dot} on the compact set $\mathcal{X}$, we find according to the Stone-Weierstrass theorem~\cite[Theorem~3, p.~24]{Mhaskar.2000} a bounded, twice continuously differentiable function $\bar{\varepsilon}_i: \mathcal{X} \rightarrow \mathbb{R}$ such that $V_i^*(\bm{x}) = {\bm{\theta}_i^*}^\top \bm{\phi}_i(\bm{x}) + \bar{\varepsilon}_i(\bm{x})$ holds for the GT value function on $\mathcal{X}$. By defining the control law $\tilde{\bm{\mu}}^*_i(\bm{x}) = -\frac{1}{2} \bm{R}_{ii}^{*-1}\bm{G}_i^\top \left(\pdv{\bm{\phi}_i}{\bm{x}}\right)^\top \bm{\theta}_i^*$, the bound~\eqref{eq:bounded_FNE_error} results for $\norm{\tilde{\bm{\mu}}^*_i(\bm{x}) - \bm{\mu}_i^*(\bm{x})}$.
	
	Now, since $\tilde{\bm{\mu}}^*_i(\bm{x})$ can be rewritten according to Lemma~\ref{lemma:linear_combination_mu_i}, Lemma~\ref{lemma:FNE_ident_offline} can be applied to identify $\tilde{\bm{\mu}}^*_i(\bm{x})$ and get $\hat{\bm{\mu}}_i(\bm{x})$ with $\hat{\bm{\mu}}_i(\bm{x})=\tilde{\bm{\mu}}^*_i(\bm{x})$, from which \eqref{eq:bounded_FNE_error} follows\comment{noticeably, the parameters $\bar{\bm{\theta}}_i^*$ are non-unique due to the non-uniqueness of the bounded error function $\bar{\varepsilon}_i(\bm{x})$; with the FNE identification QP, the parameters $\bar{\bm{\theta}}_i^*$ are determined such that the influence of the bounded error function $\bar{\varepsilon}_i(\bm{x})$ on the control law, i.e. \eqref{eq:bounded_FNE_error}, gets minimal (not a direct regression of $V_i^*$ on $\mathcal{X}$)}.
\end{proof}

If the combined identified control laws of all players constitute a FNE of a DG according to Section~\ref{sec:problem} where the cost function $\tilde{Q}_i(\bm{x}), \forall i \in \mathcal{N}$ is given by a parameterized structure with known basis functions $\bm{\psi}_i, \forall i \in \mathcal{N}$, the second step of our offline IDG method can be performed as in the error-free case to find all cost function parameters $\hat{\bm{\beta}}_i$ and $\hat{\bm{R}}_{ij}, \forall i,j \in \mathcal{N}$ that yield the identified FNE. 
\begin{assumption} \label{assm:Q_i_FNEident}
	Parameters $\tilde{\bm{R}}_{ij}$, $\tilde{\bm{\beta}}_i, \forall i,j \in \mathcal{N}$ exist such that they define with $\tilde{Q}_i(\bm{x}) = \tilde{\bm{\beta}}_i^\top \bm{\psi}_i(\bm{x})$ a DG as given in Section~\ref{sec:problem} with the unique FNE $\{\tilde{\bm{\mu}}^*_i: i \in \mathcal{N}\}$ and its corresponding value functions $\tilde{V}^*_i(\bm{x}) = {\bm{\theta}_i^*}^\top \bm{\phi}_i(\bm{x})$\comment{see assumption on $\bm{R}_{ij}^*$, $\bm{\theta}_i^*, \forall i,j \in \mathcal{N}$ in Section~\ref{sec:problem} for $\{\bm{\mu}^*_i: i \in \mathcal{N}\}$; again, uniqueness = unique value functions and unique FNE and unique parameters of parameterized structure value functions for value functions FNE and solutions Lyapunov equations; remarkably, the identified control laws can only constitute a FNE when their combination stabilizes the system on $\mathcal{X}$ such that the value functions are finite; differentiability characteristics of them hold through the differentiability characteristics of the basis functions}.
\end{assumption}

\begin{lemma} \label{lemma:HJBident_valueFct_error_offline}
	Let Assumptions~\ref{assm:R_ij}, \ref{assm:systemDynamics_basisFunctions_known}, \ref{assm:excitation_FNE_offline}, \ref{assm:excitation_HJB_offline} and \ref{assm:Q_i_FNEident} hold $\forall i \in \mathcal{N}$. Furthermore, let the basis functions $\bm{\phi}_i, \forall i \in \mathcal{N}$ be chosen according to Lemma~\ref{lemma:FNEident_valueFct_error_offline} and the identification of the FNE strategy $\tilde{\bm{\mu}}_i^*(\bm{x}), \forall i \in \mathcal{N}$ be performed by Lemma~\ref{lemma:FNEident_valueFct_error_offline} as well.
	All parameters $\hat{\bm{\alpha}}_i$, $\hat{\bm{\beta}}_i, \forall i \in \mathcal{N}$ according to the sets \eqref{eq:HJB_ident_offline_solSet_1} and \eqref{eq:HJB_ident_offline_solSet_2}, respectively, yield a FNE $\{\hat{\bm{\mu}}^*_i: i \in \mathcal{N}\}$ with $\hat{\bm{\mu}}_i^*(\bm{x}) = \hat{\bm{\mu}}_i(\bm{x}) = \tilde{\bm{\mu}}_i^*(\bm{x}), \forall i \in \mathcal{N}$\footnote{As in Theorem~\ref{theorem:HJB_ident_offline}, the arbitrary vectors $\bar{\bm{w}}^{(r)}_i$ and $\bar{\bm{w}}_i$, respectively, are chosen such that unique OC solutions are guaranteed.}. Thus, for $\hat{\bm{\mu}}_i^*(\bm{x}), \forall i \in \mathcal{N}$ resulting from the IDG solution the upper bound \eqref{eq:bounded_FNE_error} holds.
\end{lemma}
\begin{proof}
	From Lemma~\ref{lemma:FNEident_valueFct_error_offline}, the existence of control laws~$\tilde{\bm{\mu}}^*_i(\bm{x})$ with bounded error~\eqref{eq:bounded_FNE_error} to the GT ones~$\bm{\mu}^*_i(\bm{x})$ follows for all players $i \in \mathcal{N}$. If the control laws~$\tilde{\bm{\mu}}^*_i(\bm{x})$ are identified according to Lemma~\ref{lemma:FNEident_valueFct_error_offline}, we have $\hat{\bm{\mu}}_i(\bm{x}) = \tilde{\bm{\mu}}^*_i(\bm{x}), \forall i \in \mathcal{N}$ and the bounded error~\eqref{eq:bounded_FNE_error} to the GT one for both control laws. Now, if Assumption~\ref{assm:Q_i_FNEident} holds, the control laws $\tilde{\bm{\mu}}^*_i(\bm{x})$ constitute a FNE $\{\tilde{\bm{\mu}}^*_i: i \in \mathcal{N}\}$ for a DG as defined in Section~\ref{sec:problem} where the unknown cost function $\tilde{Q}_i(\bm{x})$ is given by a known parameterized structure $\tilde{Q}_i(\bm{x}) = \tilde{\bm{\beta}}_i^\top \bm{\psi}_i(\bm{x})$ with unknown parameters $\tilde{\bm{\beta}}_i$.
	
	For this DG, the coupled HJB equations can be rewritten as \eqref{eq:hjb_linear_1} or \eqref{eq:hjb_linear_2} (with $\tilde{\bm{\mu}}^*$ for $\bm{\mu}^*$, $\tilde{\bm{\alpha}}_i$ for $\bm{\alpha}^*_i$ and $\tilde{\bm{\beta}}_i$ for $\bm{\beta}^*_i$). By inserting $\hat{\bm{\mu}}_i(\bm{x})$ from Lemma~\ref{lemma:FNEident_valueFct_error_offline} for $\tilde{\bm{\mu}}_i^*(\bm{x}), \forall i \in \mathcal{N}$ in \eqref{eq:hjb_linear_1} or \eqref{eq:hjb_linear_2} (depending on $\rank\left(\bm{M}_{u_i}\right)$ in the identification of Lemma~\ref{lemma:FNEident_valueFct_error_offline}), $N$ decoupled equations follow where each equation depends only on parameters of player $i$. Following the proof of Theorem~\ref{theorem:HJB_ident_offline}, all parameters $\hat{\bm{\alpha}}_i$, $\hat{\bm{\beta}}_i, \forall i \in \mathcal{N}$ according to \eqref{eq:HJB_ident_offline_solSet_1} and \eqref{eq:HJB_ident_offline_solSet_2}, respectively, yield a FNE $\{\hat{\bm{\mu}}^*_i: i \in \mathcal{N}\}$ with $\hat{\bm{\mu}}_i^*(\bm{x}) = \hat{\bm{\mu}}_i(\bm{x}) = \tilde{\bm{\mu}}_i^*(\bm{x}), \forall i \in \mathcal{N}$. Thus, the upper bound \eqref{eq:bounded_FNE_error} holds for $\hat{\bm{\mu}}_i^*(\bm{x}), \forall i \in \mathcal{N}$ as well.
\end{proof}

\begin{remark} \label{remark:choice_basisFunction_costFunction}
	Appropriate basis functions~$\bm{\psi}_i, \forall i \in \mathcal{N}$ for the FNE~$\{\tilde{\bm{\mu}}^*_i: i \in \mathcal{N}\}$ can be designed based on the HJB equations, similarly to a converse HJB approach\cite{Nevistic.1996}. Since via the control law identification (cf.~Lemma~\ref{lemma:FNEident_valueFct_error_offline}) the structures of the FNE strategies and their corresponding value functions are given, the basis functions~$\bm{\psi}_i$ for $\tilde{\bm{Q}}_i(\bm{x})$ need to be chosen such that the HJB equations~\eqref{eq:hjb} (with $\tilde{V}_i^*$ for $V_i^*$ and $\tilde{\bm{\mu}}_i^*$ for $\bm{\mu}_i^*$ for all $i \in \mathcal{N}$) can be fulfilled $\forall \bm{x} \in \mathcal{X}$ with appropriate parameters~$\tilde{\bm{\beta}}_i$.
\end{remark}

\comment{Lemmas~\ref{lemma:FNEident_valueFct_error_offline} and \ref{lemma:HJBident_valueFct_error_offline} explain how an approximate LQ DG can be designed: assumption on quadratic value function - linear control laws - best linear control laws to explain GT FNE via control law identification with Lemma~\ref{lemma:FNEident_valueFct_error_offline} - bounded error to GT FNE (slightly different to \eqref{eq:bounded_FNE_error}, since $\bm{G}_i$ and $\bm{B}_i$ different in general - we use in case of an approximate LQ DG other system dynamics than the GT dynamics) - if linear control laws constitute unique linear FNE of LQ DG with quadratic cost functions (in general not guaranteed resp. not shown in literature so far), all equivalent cost function parameters that yield this linear FNE given by Lemma~\ref{lemma:HJBident_valueFct_error_offline}}

\comment{Lemmas~\ref{lemma:FNEident_valueFct_error_offline} and \ref{lemma:HJBident_valueFct_error_offline} explain that IDG methods based on FNE identification can be more robust than approaches directly using trajectories; for example, noisy GT trajectories - control law identification yields still underlying FNE, which produced the GT trajectories with added noise}

\subsection{Approximated Cost Function Structures} \label{subsec:IDGHJB_offline_approx_costFct_structures}
In this section, we analyze the case where Assumption~\ref{assm:Q_i} or \ref{assm:Q_i_FNEident} is not fulfilled, i.e. the cost functions $Q_i^*$ (or $\tilde{Q}_i$) for a given/identified FNE $\{\bm{\mu}^*_i: i \in \mathcal{N}\}$ (or $\{\tilde{\bm{\mu}}^*_i: i \in \mathcal{N}\}$) are not given by parameterized structures with known basis functions $\bm{\psi}_i, \forall i \in \mathcal{N}$. Instead, these basis functions are chosen according to an approximation theorem. In our analysis, without loss of generality, we again choose $\bm{\psi}_i, \forall i \in \mathcal{N}$ in line with the Stone-Weierstrass theorem. In contrast to the value function approximation, the cost function approximation is the more sensitive part of our new IDG method (similar results were found for other IDG methods \cite{Karg.2024}). Although we can again find a bounded error function $\varepsilon_i: \mathcal{X} \rightarrow \mathbb{R}$ such that $Q_i^*(\bm{x}) = \bm{\beta}_i^{*\top} \bm{\psi}_i(\bm{x}) + \varepsilon_i(\bm{x})$ (or $\tilde{Q}_i(\bm{x}) = \bm{\beta}_i^{*\top} \bm{\psi}_i(\bm{x}) + \varepsilon_i(\bm{x})$) holds $\forall \bm{x} \in \mathcal{X}$, this remaining unknown approximation error leads to a bias of the parameters $\hat{\bm{\alpha}}_i, \hat{\bm{\theta}}^*_i$ to the searched parameters $\bm{\alpha}^*_i, \bm{\theta}^*_i$ (or $\tilde{\bm{\alpha}}_i, \bm{\theta}^*_i$). Then, the FNE~$\{\hat{\bm{\mu}}^*_i: i \in \mathcal{N}\}$ resulting from the biased parameters $\hat{\bm{\alpha}}_i, \hat{\bm{\theta}}^*_i, \forall i \in \mathcal{N}$ is in general unequal to the given/identified FNE~$\{\bm{\mu}^*_i: i \in \mathcal{N}\}$ (or $\{\tilde{\bm{\mu}}^*_i: i \in \mathcal{N}\}$). In contrast to approximation errors in the FNE identification step, the error between both FNEs cannot be quantified and thus, no bound can be guaranteed in general. In Lemma~\ref{lemma:costFct_error_offline}, this result is given formally\footnote{For brevity, only the case where Assumption~\ref{assm:V_i_star} is fulfilled, i.e. the FNE~$\{\bm{\mu}^*_i: i \in \mathcal{N}\}$ is given/identified and $Q^*_i(\bm{x})$ is approximated, and $\rank\left(\bm{M}_{u_i}\right)<p_i\bar{h}_i$ is discussed. The result extends to the case where the FNE~$\{\tilde{\bm{\mu}}^*_i: i \in \mathcal{N}\}$ is given/identified and $\tilde{Q}_i(\bm{x})$ is approximated as well as to $\rank\left(\bm{M}_{u_i}\right)=p_i\bar{h}_i$ in both cases.}.

\begin{lemma} \label{lemma:costFct_error_offline}
	Let Assumptions~\ref{assm:R_ij}, \ref{assm:V_i_star}, \ref{assm:systemDynamics_basisFunctions_known}, \ref{assm:excitation_FNE_offline} and \ref{assm:excitation_HJB_offline} hold $\forall i \in \mathcal{N}$. If Assumption~\ref{assm:Q_i} is not fulfilled but the basis functions~$\bm{\psi}_i, \forall i \in \mathcal{N}$ are from a subalgebra (cf.~\cite[pp.~22f.]{Mhaskar.2000}) of the set of bounded and continuously differentiable functions on $\mathcal{X}$, a bounded, continuously differentiable function $\varepsilon_i: \mathcal{X} \rightarrow \mathbb{R}$ exists such that $\forall \bm{x} \in \mathcal{X}$ $Q_i^*(\bm{x}) = \bm{\beta}_i^{*\top} \bm{\psi}_i(\bm{x}) + \varepsilon_i(\bm{x})$, where the parameters $\bm{\alpha}_i^*$, $\bm{\beta}_i^*$ and $\bm{\theta}_i^*$ fulfill $\forall i \in \mathcal{N}$
	\begin{align} \label{eq:hjb_error}
		\underbrace{\mat{ \bm{\mu}^{*\top} \odot \bm{\mu}^{*\top} & \!\! \bm{\psi}_i^\top & \!\!\! \left( \pdv{\bm{\phi}_i}{\bm{x}} \bm{f}^*_g \right)^\top }}_{{\bm{m}_{\hjb_i}(\bm{x})}^\top} \mat{ \bm{\alpha}^*_i \\ \bm{\beta}^*_i \\ \bm{\theta}^*_i } \! = \!\!\!\! \underbrace{0}_{z_{\hjb_i}(\bm{x})} \!\!\!\!\! - \varepsilon_i(\bm{x}).
	\end{align}
	Let the parameters $\hat{\bm{\alpha}}_i, \hat{\bm{\beta}}_i, \forall i \in \mathcal{N}$ (note that $\hat{\bm{\beta}}_i = \bm{\beta}_i^*$, since $\bm{\beta}_i^*$ and $\varepsilon_i$ are only defined after the HJB-based QP) be determined by \eqref{eq:HJB_ident_offline_solSet_2}\comment{here, the arbitrary vectors to compute $\hat{\bm{\alpha}}_i, \hat{\bm{\beta}}_i, \forall i \in \mathcal{N}$ cannot be defined such that unique OC solutions are guaranteed; instead, it needs to be assumed that the computed parameters with the cost functions defined only with the chosen basis functions $\bm{\psi}_i$ define a DG that has a unique FNE with corresponding value functions $\hat{V}^*_i(\bm{x})=\hat{\bm{\theta}}_i^{*\top}\bm{\phi}_i(\bm{x})$, i.e. the value functions are also linear combinations of the given value functions basis functions}, where the $\bm{\mu}_i^*(\bm{x}), \forall i \in \mathcal{N}$ are identified according to Lemma~\ref{lemma:FNE_ident_offline}, and let a DG as given in Section~\ref{sec:problem} with $\hat{Q}_i(\bm{x}) = \hat{\bm{\beta}}_i^\top \bm{\psi}_i(\bm{x}), \forall i \in \mathcal{N}$ be defined such that it has a unique FNE~$\{\hat{\bm{\mu}}^*_i: i \in \mathcal{N}\}$\comment{again, unique value functions, unique FNE, unique parameters value function structures for all LEs} with corresponding value functions $\hat{V}^*_i(\bm{x})=\hat{\bm{\theta}}_i^{*\top}\bm{\phi}_i(\bm{x})$\comment{in the cases before, the representation of the value functions, corresponding to the FNE resulting from the IDG solution, as linear combination of the basis functions used in the FNE identification step was given implicitly by the method since cost function parameters are chosen to guarantee unique OC solutions which correspond then exactly to unique value functions with unique parameters of the used linear combinations; remarkably, again, no exact approximation of value functions in LEs during PI algorithm guaranteed}. Then, $\bm{\mu}_i^*(\bm{x}) \neq \hat{\bm{\mu}}^*_i(\bm{x}), \forall i \in \mathcal{N}$ in general.
\end{lemma}
\begin{proof}
	According to the Stone-Weierstrass theorem~\cite[Theorem~3, p.~24]{Mhaskar.2000} a bounded, continuously differentiable function $\varepsilon_i: \mathcal{X} \rightarrow \mathbb{R}$ exists such that $Q_i^*(\bm{x}) = \bm{\beta}_i^{*\top} \bm{\psi}_i(\bm{x}) + \varepsilon_i(\bm{x})$ holds on $\mathcal{X}$ $\forall i \in \mathcal{N}$. Thus, \eqref{eq:hjb} can be written as \eqref{eq:hjb_error}.
	
	Now, in order to estimate the parameters $\bm{\alpha}^*_i, \bm{\beta}^*_i, \forall i \in \mathcal{N}$, we would need to define the QP~\eqref{eq:proof_HJB_ident_offline_2} of the HJB identification step with 
	$\bm{\epsilon}_{\hjb_i} = \bm{M}_{\hjb_i}\bm{\eta}_i-\bm{z}_{\hjb_i}-\bm{\varepsilon}_i$, where $\bm{\varepsilon}_i$ is the concatenation of the evaluations of the error function $\varepsilon_i(\bm{x})$ on the data points $\bm{x}_{\bar{k}} \in \mathcal{X}$. However, since this error function is unknown, with the parameters~$\hat{\bm{\eta}}_i$ computed by \eqref{eq:HJB_ident_offline_solSet_2} a bias results:
	\begin{align} \label{eq:proof_costFct_error_offline_1}
		\bm{\eta}_i^* = \hat{\bm{\eta}}_i + \bar{\bm{M}}^+_{\hjb_i} \bm{M}_{\hjb_i}^\top \bm{\varepsilon}_i.
	\end{align}
	\comment{in general, $\bm{\eta}_i^*$ and $\hat{\bm{\eta}}_i$ result from different QPs, where the QP of the first one is guaranteed to have a global minimum value of zero; the second one has in general a non-zero global minimum value, a zero global minimum value would lead to an exact solution of the LE in the policy evaluation of the FNE $\{\bm{\mu}^*_i: i \in \mathcal{N}\}$ with the chosen basis functions $\bm{\phi}_i$ with the parameters $\hat{\bm{\theta}}_i$, which are due to the bias~\eqref{eq:proof_costFct_error_offline_1} unequal to $\bm{\theta}^*_i$; independent from the global minimum value of this QP, in the policy evaluation of $\{\bm{\mu}^*_i: i \in \mathcal{N}\}$, $\hat{\bm{\theta}}_i$ (with the bias to $\bm{\theta}^*_i$) results uniquely (PI algorithm constructs unique sequence of value functions/LE solutions and we assume appropriate basis functions for the value functions such that this behavior transfers to unique LS parameters of the parameterized value functions) and we have $\hat{\bm{\theta}}_i \neq \bm{\theta}^*_i$ (and from this a policy improvement and thus, $\{\bm{\mu}^*_i: i \in \mathcal{N}\}$ as not a searched fix point of the PI); finally, also $\hat{\bm{\theta}}^*_i \neq \bm{\theta}^*_i$ in the converged state of PI, where however the LE in the policy evaluation of the FNE $\{\hat{\bm{\mu}}^*_i: i \in \mathcal{N}\}$ with the chosen basis functions $\bm{\phi}_i$ can be solved exactly with the parameters $\hat{\bm{\theta}}^*_i$ (see assumption in lemma)}
	
	If the determined parameters $\hat{\bm{\alpha}}_i, \hat{\bm{\beta}}_i, \forall i \in \mathcal{N}$ define a DG with $\hat{Q}_i(\bm{x})=\hat{\bm{\beta}}^\top\bm{\psi}_i(\bm{x})$ with a unique FNE~$\{\hat{\bm{\mu}}^*_i: i \in \mathcal{N}\}$ and corresponding value functions~$\hat{V}^*_i(\bm{x})=\hat{\bm{\theta}}_i^{*\top}\bm{\phi}_i(\bm{x})$, a PI algorithm initialized with an arbitrary stabilizing policy combination converges to this FNE. Using the GT FNE~$\{\bm{\mu}^*_i: i \in \mathcal{N}\}$ as initializing policy combination, we need to solve the Lyapunov equations
	\begin{align} \label{eq:proof_costFct_error_offline_2}
		\underbrace{\mat{ \bm{\mu}^{*\top} \odot \bm{\mu}^{*\top} & \bm{\psi}_i^\top & \left( \pdv{\bm{\phi}_i}{\bm{x}} \bm{f}^*_g \right)^\top }}_{{\bm{m}_{\hjb_i}(\bm{x})}^\top} \mat{ \hat{\bm{\alpha}}_i \\ \hat{\bm{\beta}}_i \\ \bm{\theta}_i } = \underbrace{0}_{z_{\hjb_i}(\bm{x})}, \forall i \in \mathcal{N}
	\end{align}
	for $\bm{\theta}_i, \forall i \in \mathcal{N}$ in the first policy evaluation step. Evaluating the Lyapunov equations on the same grid points as the HJB equations to compute \eqref{eq:HJB_ident_offline_solSet_2}, we need to solve 
	\begin{align} \label{eq:proof_costFct_error_offline_3}
		\bm{M}_{\hjb_i} \mat{ \hat{\bm{\alpha}}_i \\ \hat{\bm{\beta}}_i \\ \bm{\theta}_i } = \bm{z}_{\hjb_i}, \forall i \in \mathcal{N}
	\end{align}
	for $\bm{\theta}_i, \forall i \in \mathcal{N}$. Remarkably, it is not guaranteed that the Lyapunov equations \eqref{eq:proof_costFct_error_offline_2} and \eqref{eq:proof_costFct_error_offline_3} can be solved exactly. Hence, by setting up a QP from \eqref{eq:proof_costFct_error_offline_3} with $\bm{\epsilon}_{\hjb_i} = \bm{M}_{\hjb_i}\bm{\eta}_i\rvert_{\bm{\alpha}_i=\hat{\bm{\alpha}}_i, \bm{\beta}_i=\hat{\bm{\beta}}_i}-\bm{z}_{\hjb_i}$ to determine $\bm{\theta}_i$ we get the parameter $\hat{\bm{\theta}}_i$ from the set \eqref{eq:HJB_ident_offline_solSet_2}\comment{unique connection between cost and value function parameters; follows from assumption on unique FNE (incl. unique value functions) resp. unique sequence value functions/LE solutions and implicit assumption on basis functions for value functions that this uniqueness transfers to the parameters for these basis functions}. From \eqref{eq:proof_costFct_error_offline_1}, $\hat{\bm{\theta}}_i \neq \bm{\theta}_i^*$ results and thus, in general, a policy improvement follows, i.e. the GT FNE~$\{\bm{\mu}^*_i: i \in \mathcal{N}\}$ is not the searched fix point~$\{\hat{\bm{\mu}}^*_i: i \in \mathcal{N}\}$ of the PI algorithm. Finally, we conclude that $\bm{\mu}_i^*(\bm{x}) \neq \hat{\bm{\mu}}^*_i(\bm{x}), \forall i \in \mathcal{N}$, i.e. the GT FNE~$\{\bm{\mu}^*_i: i \in \mathcal{N}\}$ and the FNE~$\{\hat{\bm{\mu}}^*_i: i \in \mathcal{N}\}$ resulting from the estimated parameters~$\hat{\bm{\alpha}}_i, \hat{\bm{\beta}}_i, \forall i \in \mathcal{N}$ do not match, in general. 
\end{proof}

In summary, for our new IDG method, it is important to guarantee either a sufficiently good approximation structure for completely unknown cost functions or to design appropriate cost function structures manually based on Remark~\ref{remark:choice_basisFunction_costFunction}.

\comment{here, it gets clear that the effects of erroneous value and cost function approximations can combine and thus, even higher FNE and trajectory errors are possible}

\comment{the proof of Lemma~\ref{lemma:costFct_error_offline} shows that if the error function is small or if the influence of the error function is small at the data points used for evaluation of HJB, the bias is small and the error between GT FNE and FNE resulting from the IDG solution is small as well - see online vs. offline result erroneous case simulation example, when probing noise is used for HJB evaluation online method: smaller influence error function in online case and thus, smaller FNE error}

\section{Online Inverse Differential Game Method Based on Hamilton-Jacobi-Bellman Equations} \label{sec:onlineIDGHJB}
In this section, we propose a new online IDG method to solve Problem~\ref{problem:online_idg} based on the theoretical framework of the offline method presented in the previous section. 

\subsection{Known Cost and Value Function Structures} \label{subsec:IDGHJB_online_known_structures}
\tdd{lemma for online FNE identification, second main theorem for final IDG solution to second problem}

Similarly to the logic of the offline approach, we formulate for each player $i \in \mathcal{N}$ an adaption scheme for the parameters $\bar{\bm{\theta}}^*_i$ to identify \eqref{eq:mu_i_linear}. This is done by solving the QP \eqref{eq:proof_FNE_ident_offline_2} via an online gradient descent. With $\bm{M}_{u_i}(t)=-\frac{1}{2}\bar{\bm{\Phi}}_i(\bm{x}^*(t))$ and $\bm{z}_{u_i}(t)=\bm{u}^*_i(t)$, we define $\bm{\epsilon}_{u_i}(t)=\bm{M}_{u_i}(t)\bar{\bm{\theta}}_i-\bm{z}_{u_i}(t)$ and
\begin{align} \label{eq:FNElearninglaw}
	\dot{\hat{\bar{\bm{\theta}}}}_i(t) &= -\tau_i \frac{1}{2}\left(\pdv{\bm{\epsilon}^\top_{u_i}(t)\bm{\epsilon}_{u_i}(t)}{\bar{\bm{\theta}}_i}\right)^\top \notag \\
	&= - \tau_i \bm{M}^\top_{u_i}(t) \bm{M}_{u_i}(t) \hat{\bar{\bm{\theta}}}_i(t) + \tau_i\bm{M}^\top_{u_i}(t) \bm{z}_{u_i}(t),
\end{align}
where $\tau_i \in \mathbb{R}_{>0}$ is the learning rate. Starting from an initial guess $\hat{\bar{\bm{\theta}}}_i(t_0)$ the adaption according to \eqref{eq:FNElearninglaw} is stopped when 
\begin{align} \label{eq:stoppingCriterion}
	\frac{1}{T}\norm{\int_{t-T}^{t}\hat{\bar{\bm{\theta}}}_i(\tau)\text{d}\tau - \int_{t-2T}^{t-T}\hat{\bar{\bm{\theta}}}_i(\tau)\text{d}\tau}
\end{align}
is small\comment{error as well as weights can oscillate, $T$ here actually the $T$ from PE definition, in period of $T$ parameter error descent guaranteed (in general, only for the "excited" parameters) - in simulation, additional stopping criterion analogously defined with error function in place}. Lemma~\ref{lemma:FNEident_online} shows that under a common excitation condition (cf.~Assumption~\ref{assm:excitation_FNEident_online}) on the regression matrix $\bm{M}_{u_i}(t)$ the parameters $\hat{\bar{\bm{\theta}}}_i$ converge to $\bar{\bm{\theta}}^*_i$ and thus, also the estimated FNE strategy $\hat{\bm{\mu}}_i(\bm{x})$ to the GT one $\bm{\mu}^*_i(\bm{x})$\comment{possible to state partial convergence as well, under largest possible PE in subspace, results should extend/apply similar to HJB identification, i.e. then convergence to solution set of LS-based offline identification, but no statement about weight bound in error case possible (but weight bound in non-error case still guaranteed) and subspace PE statements in literature only clearly described for regression vectors and not regression matrices (would be necessary for FNE identification here)}.

\begin{definition}[cf.~{\cite[Definition~6.2, Definition~6.3]{Narendra.2005}}] \label{definition:PE}
	A piecewise continuously differentiable, bounded function (cf.~\cite[Definition~6.1]{Narendra.2005}) $\bm{M}: \mathbb{R}_{\geq 0} \rightarrow \mathbb{R}^{m\times h}$ is called persistently excited (PE) for all $t\geq t_0$ if there exist positive constants $\gamma, T \in \mathbb{R}_{>0}$ such that $\int_{t}^{t+T} \bm{M}(\tau)\bm{M}^\top(\tau) \text{d}\tau \succeq \gamma \bm{I}$ for all $t \geq t_0$. Let $h=1$. Then, the function $\bm{M}(t)$ is called PE in a subspace of $\mathbb{R}^{m}$ if there exists a matrix $\bm{T}_1 \in \mathbb{R}^{r\times m}$ with $r \leq m$ such that $\bm{T}_1\bm{M}(t)$ is PE. 
\end{definition}

\begin{assumption} \label{assm:excitation_FNEident_online}
	The regression matrix $\bm{M}^\top_{u_i}(t)$ is PE\footnote{Reformulating the FNE strategy~$\bm{\mu}_i^*(\bm{x})$ according to \eqref{eq:mu_i_red} is necessary to fulfill this excitation condition.}.
\end{assumption}

\begin{lemma} \label{lemma:FNEident_online}
	If Assumptions~\ref{assm:R_ij}, \ref{assm:V_i_star}, \ref{assm:systemDynamics_basisFunctions_known} and \ref{assm:excitation_FNEident_online} hold, the parameters~$\hat{\bar{\bm{\theta}}}_i$ adapted by the learning scheme~\eqref{eq:FNElearninglaw} exponentially converge from an arbitrary starting point $\hat{\bar{\bm{\theta}}}_i(t_0) \in \mathbb{R}^{p_i \bar{h}_i}$ to $\bar{\bm{\theta}}^*_i$ and thus, also $\hat{\bm{\mu}}_i(\bm{x})$ to $\bm{\mu}_i^*(\bm{x})$.
\end{lemma}
\begin{proof}
	From Lemma~\ref{lemma:linear_combination_mu_i}, we have $\bm{u}_i^*(t) = \bm{M}_{u_i}(t)\bar{\bm{\theta}}^*_i, \forall t\geq t_0$. By defining $\bar{\bm{\theta}}^{(\text{err})}_i(t) = \hat{\bar{\bm{\theta}}}_i(t) - \bar{\bm{\theta}}_i^*$, we get the error dynamics
	\begin{align} \label{eq:proof_FNEident_online_1}
		\dot{\bar{\bm{\theta}}}^{(\text{err})}_i(t) &= - \tau_i \bm{M}^\top_{u_i}(t) \bm{M}_{u_i}(t) \left(\bar{\bm{\theta}}^{(\text{err})}_i(t) + \bar{\bm{\theta}}_i^* \right) \notag \\
		&\hphantom{=}+ \tau_i\bm{M}^\top_{u_i}(t) \bm{z}_{u_i}(t) = - \tau_i \bm{M}^\top_{u_i}(t) \bm{M}_{u_i}(t) \bar{\bm{\theta}}^{(\text{err})}_i(t).
	\end{align}
	According to \cite{Jenkins.2018}\comment{formally, only with a regression vector, but they directly follow the initial proof of Morgan/Narendra 1977, which considers the more general case of regression matrices}, the equilibrium $\bar{\bm{\theta}}^{(\text{err})}_i = \bm{0}$ of \eqref{eq:proof_FNEident_online_1} is exponentially stable in the large if Assumption~\ref{assm:excitation_FNEident_online} holds. Thus, from this stability property, we get $\hat{\bar{\bm{\theta}}}_i \rightarrow \bar{\bm{\theta}}^*_i$ and $\hat{\bm{\mu}}_i(\bm{x}) \rightarrow \bm{\mu}^*_i(\bm{x})$ exponentially fast.
\end{proof}

To infer the cost function parameters $\hat{\bm{\alpha}}_i$ and $\hat{\bm{\beta}}_i, \forall i \in \mathcal{N}$, for each player an additional adaption scheme is set up based on solving the QP~\eqref{eq:proof_HJB_ident_offline_2} by an online gradient descent. Firstly, the current estimates $\hat{\bar{\bm{\theta}}}_i(t), \forall i \in \mathcal{N}$ are used to derive current estimates of the FNE strategies $\hat{\bm{\mu}}_i(\bm{x})$ (cf.~\eqref{eq:mu_i_hat}) and the value function weights $\hat{\bm{\theta}}_i^{(r)}$ (cf.~\eqref{eq:theta_r_hat}), which are then used to set up $\bm{m}_{\hjb_i}^{(r)}(\bm{x})^\top$ and $z_{\hjb_i}^{(r)}(\bm{x})$ (cf.~\eqref{eq:hjb_linear_1}). Then, $\bm{m}_{\hjb_i}^{(r)}(\bm{x})^\top$ and $z_{\hjb_i}^{(r)}(\bm{x})$ are evaluated for the current measurement $\bm{x}^*(t)$: $\bm{m}_{\hjb_i}^{(r)\top}(t) = \bm{m}_{\hjb_i}^{(r)}(\bm{x}^*(t))^\top$, $z_{\hjb_i}^{(r)}(t)= z_{\hjb_i}^{(r)}(\bm{x}^*(t))$. Finally, the error function $\epsilon_{\hjb_i}^{(r)}(t) = \bm{m}_{\hjb_i}^{(r)\top}(t) \bm{\eta}_i^{(r)} - z_{\hjb_i}^{(r)}(t)$ can be defined as well as the adaption law
\begin{align} \label{eq:HJB_learningLaw}
	\dot{\hat{\bm{\eta}}}_i^{(r)}(t) &= -\kappa_i \frac{1}{2} \left(\pdv{\epsilon_{\hjb_i}^{(r)}(t)\epsilon_{\hjb_i}^{(r)}(t)}{\bm{\eta}_i^{(r)}} \right)^\top \notag \\
	&= -\kappa_i \bm{m}_{\hjb_i}^{(r)}(t)\bm{m}_{\hjb_i}^{(r)\top}(t) \hat{\bm{\eta}}_i^{(r)}(t) \notag \\
	&\hphantom{=}+ \kappa_i \bm{m}_{\hjb_i}^{(r)}(t)z_{\hjb_i}^{(r)}(t),
\end{align}
where $\kappa_i \in \mathbb{R}_{>0}$ is the learning rate. The adaption is stopped when \eqref{eq:stoppingCriterion} for $\hat{\bm{\eta}}_i^{(r)}$ is small. Theorem~\ref{theorem:HJBident_online} shows that under the excitation Assumption~\ref{assm:excitation_HJBident_online} the parameters $\hat{\bm{\eta}}_i^{(r)}$ converge to a fixed element of the solution set of Theorem~\ref{theorem:HJB_ident_offline}, i.e. \eqref{eq:HJB_ident_offline_solSet_1}, and thus, solve Problem~\ref{problem:online_idg} when for the converged value $\hat{\bm{\eta}}_i^{(r)}$ an arbitrary vector $\bar{\bm{w}}_i^{(r)}$ in line with Theorem~\ref{theorem:HJB_ident_offline} exists\comment{the additional constraints on $\hat{\bm{\alpha}}_i, \hat{\bm{\beta}}_i$ need to be applied online/during learning; in case of simple upper/lower bounds easily implementable}.

\begin{assumption} \label{assm:excitation_HJBident_online}
	The vector $\bm{m}_{\hjb_i}^{(r)}(t)$ is PE in the largest possible subspace of $\mathbb{R}^{p+m_i+h_i-\bar{h}_i}$, i.e. there is no matrix $\bm{T}_{i_1}$ with a higher rank such that $\bm{T}_{i_1} \bm{m}_{\hjb_i}^{(r)}(t)$ is PE\comment{full PE in most cases not given since typically non-uniqueness of IDG solutions beyond scaling ambiguity}.
\end{assumption}

\begin{theorem} \label{theorem:HJBident_online}
	Let Assumptions~\ref{assm:Q_i}, \ref{assm:R_ij}, \ref{assm:V_i_star}, \ref{assm:systemDynamics_basisFunctions_known}, \ref{assm:excitation_HJB_offline}, \ref{assm:excitation_FNEident_online} and \ref{assm:excitation_HJBident_online} hold $\forall i \in \mathcal{N}$. Furthermore, let the identification of the FNE strategy $\bm{\mu}^*_i(\bm{x}), \forall i \in \mathcal{N}$ be performed by Lemma~\ref{lemma:FNEident_online}. Then, for every player $i \in \mathcal{N}$ the adaption scheme~\eqref{eq:HJB_learningLaw} applied to arbitrary initial values $\hat{\bm{\eta}}_i^{(r)}(t_0) \in \mathbb{R}^{p+m_i+h_i-\bar{h}_i}$ yields convergence\comment{exponential convergence not guaranteed (but reasonable since only PE in subspace)} of $\hat{\bm{\eta}}_i$ to a fixed element of the solution set~\eqref{eq:HJB_ident_offline_solSet_1}.
\end{theorem}
\begin{proof}
	According to Lemma~\ref{lemma:FNEident_online}, the FNE identification $\hat{\bm{\mu}}_i(\bm{x})$ converges exponentially fast to $\bm{\mu}^*_i(\bm{x}), \forall i \in \mathcal{N}$. Thus, we find a time $t_{\text{conv}} \geq t_0$ where $\hat{\bm{\mu}}_i(\bm{x}) \approx \bm{\mu}^*_i(\bm{x}), \forall i \in \mathcal{N}$. Now, we define $\forall t \geq t_{\text{conv}}$ with $\bm{\eta}_i^{(r,\text{err})}(t) = \hat{\bm{\eta}}_i^{(r)}(t) - \bm{\eta}_i^{(r)^*}$ the error dynamics
	\begin{align} \label{eq:proof_HJBident_online_1}
		\dot{\bm{\eta}}_i^{(r,\text{err})}(t) = -\kappa_i \bm{m}_{\hjb_i}^{(r)}(t)\bm{m}_{\hjb_i}^{(r)\top}(t) \bm{\eta}_i^{(r,\text{err})}(t),
	\end{align}
	where we used $\bm{m}_{\hjb_i}^{(r)\top}(t) \bm{\eta}_i^{(r)^*} = z_{\hjb_i}^{(r)}(t), \forall t\geq t_{\text{conv}}$ (cf.~\eqref{eq:hjb_linear_1}). 
	
	Let $\bm{T}_i\in\mathbb{R}^{p+m_i+h_i-\bar{h}_i \times p+m_i+h_i-\bar{h}_i}$\comment{$\bm{T}$ has not necessarily to be a matrix with the unit vectors in its columns in general; thus, it depends on how $\bm{T}_{i_1}$ can be chosen whether some elements of $\hat{\bm{\eta}}_i^{(r)}$ really converge to the corresponding elements of $\bm{\eta}_i^{(r)^*}$; but then, although $\bm{T}_{i_1}$ is chosen based on Assumption~\ref{assm:excitation_HJBident_online}, $\bm{T}_{i_2}$ is arbitrary to some extent} be an orthogonal matrix such that $\bm{T}_i^\top=\mat{ \bm{T}_{i_1}^\top & \bm{T}_{i_2}^\top }$ where $\bm{T}_{i_1}$ is the matrix according to Assumption~\ref{assm:excitation_HJBident_online}.
	Following the proof of \cite[Theorem~6.1]{Narendra.2005}, we firstly derive with $\bm{y}_{i_1}(t) = \bm{T}_{i_1}\bm{\eta}_i^{(r,\text{err})}(t)$, $\bm{y}_{i_2}(t) = \bm{T}_{i_2}\bm{\eta}_i^{(r,\text{err})}(t)$, $\bm{v}_{i_1}(t) = \bm{T}_{i_1} \bm{m}_{\hjb_i}^{(r)}(t)$ and $\bm{v}_{i_2}(t) = \bm{T}_{i_2} \bm{m}_{\hjb_i}^{(r)}(t)$ from \eqref{eq:proof_HJBident_online_1}
	\begin{align} 
		\dot{\bm{y}}_{i_1} = - \kappa_i \bm{v}_{i_1}\bm{v}_{i_1}^\top \bm{y}_{i_1} - \kappa_i \bm{v}_{i_1}\bm{v}_{i_2}^\top \bm{y}_{i_2}, \label{eq:proof_HJBident_online_2} \\
		\dot{\bm{y}}_{i_2} = - \kappa_i \bm{v}_{i_2}\bm{v}_{i_1}^\top \bm{y}_{i_1} - \kappa_i \bm{v}_{i_2}\bm{v}_{i_2}^\top \bm{y}_{i_2}, \label{eq:proof_HJBident_online_3}
	\end{align}
	where $\bm{y}_{i_1}$, $\bm{y}_{i_2}$, $\bm{v}_{i_1}$ and $\bm{v}_{i_2}$ are bounded, which follows from the boundedness of $\bm{\eta}_i^{(r,\text{err})}$ and $\bm{m}_{\hjb_i}^{(r)}$, respectively.
	From Assumption~\ref{assm:excitation_HJBident_online}, it follows that $\bm{v}_{i_1}$ is PE and $\norm{\bm{v}_{i_2}}_1\rightarrow 0$ as $t\rightarrow \infty$ (cf.~\cite[Lemma~6.4]{Narendra.2005}). According to \cite[Theorem~6.1]{Narendra.2005}, $\bm{y}_{i_1} \rightarrow \bm{0}$ as $t \rightarrow \infty$ and thus, also $\lim_{t\rightarrow \infty} \norm{\dot{\bm{y}}_{i_1}}_1 = \lim_{t\rightarrow \infty} \abs{\kappa_i}\norm{\bm{v}_{i_1}}_1\abs{\bm{v}_{i_2}^\top\bm{y}_{i_2}} = 0$ due to the PE of $\bm{v}_{i_1}$. Since $\bm{v}_{i_1}$ is PE, we further have $\norm{\bm{v}_{i_1}}_1 \nrightarrow 0$ as $t\rightarrow \infty$ and hence, $\abs{\bm{v}_{i_2}^\top\bm{y}_{i_2}} \rightarrow 0$ as $t\rightarrow \infty$. This yields with $\norm{\dot{\bm{y}}_{i_2}}_1 \leq \abs{\kappa_i}\norm{\bm{v}_{i_2}}_1 \abs{\bm{v}_{i_1}^\top \bm{y}_{i_1}} + \abs{\kappa_i}\norm{\bm{v}_{i_2}}_1\abs{\bm{v}_{i_2}^\top \bm{y}_{i_2}}$ $\lim_{t \rightarrow \infty} \norm{\dot{\bm{y}}_{i_2}}_1 = 0$ since $\norm{\bm{v}_{i_2}}_1 \rightarrow 0$ as $t\rightarrow \infty$. We have established the convergence of $\hat{\bm{\eta}}_i$ to a fixed value ($\dot{\bm{y}}_{i_1} \rightarrow 0$ and $\dot{\bm{y}}_{i_2} \rightarrow 0$ as $t\rightarrow \infty$) and the convergence of $\bm{T}_{i_1}\hat{\bm{\eta}}_i^{(r)}$ to $\bm{T}_{i_1}\bm{\eta}_i^{(r)^*}$.
	
	Furthermore, from $\abs{\bm{v}_{i_2}^\top\bm{y}_{i_2}} \rightarrow 0$ as $t\rightarrow \infty$, we conclude that for the converged value $\bm{\eta}_i^{(r,\text{err})}$ $\bm{m}_{\hjb_i}^{(r)\top}(t) \bm{T}_{i_2}^\top \bm{T}_{i_2} \bm{\eta}_i^{(r,\text{err})} = 0$ holds $\forall t \geq t_{\text{conv}}$ large enough. Now, we can assume that the GT trajectory $\bm{x}^*(t_0 \rightarrow \infty)$ is excited such that Assumption~\ref{assm:excitation_HJB_offline} is fulfilled as well\comment{here, it gets clear that additional excitation signals according to Remark~\ref{remark:excitation_FNEvsHJB} can be very beneficial/even necessary} and we get $\bm{M}_{\hjb_i}^{(r)} \bm{T}_{i_2}^\top \bm{T}_{i_2} \bm{\eta}_i^{(r,\text{err})} = 0$ for the converged value $\bm{\eta}_i^{(r,\text{err})}$. Finally, we can show that the converged values $\hat{\bm{\eta}}_i^{(r)} = \bm{T}^\top_{i_1}\bm{T}_{i_1}\hat{\bm{\eta}}_i^{(r)} + \bm{T}^\top_{i_2}\bm{T}_{i_2}\hat{\bm{\eta}}_i^{(r)}$ fulfill \eqref{eq:proof_HJB_ident_offline_1}
	\begin{align} \label{eq:proof_HJBident_online_4}
		&\bm{M}_{\hjb_i}^{(r)} \hat{\bm{\eta}}_i^{(r)} = \bm{M}_{\hjb_i}^{(r)}\bm{T}^\top_{i_1}\bm{T}_{i_1}\bm{\eta}_i^{(r)^*} + \bm{M}_{\hjb_i}^{(r)}\bm{T}^\top_{i_2}\bm{T}_{i_2}\bm{\eta}_i^{(r)^*} \notag \\
		&\hphantom{=}+ \bm{M}_{\hjb_i}^{(r)}\bm{T}^\top_{i_2}\bm{T}_{i_2}\bm{\eta}_i^{(r,\text{err})} = \bm{M}_{\hjb_i}^{(r)}\bm{\eta}_i^{(r)^*} = \bm{z}_{\hjb_i}^{(r)}
	\end{align}
	and are thus from the set \eqref{eq:HJB_ident_offline_solSet_1}.
\end{proof}

For reproducibility, Algorithm~\ref{alg:onlineIDGHJB} summarizes the procedure of our new online IDG method.

\begin{remark} \label{remark:excitation_FNEvsHJB}
	Whereas the excitation Assumption~\ref{assm:excitation_FNEident_online} for the FNE strategy identification needs to be fulfilled solely by the demonstrated GT trajectories $\bm{x}^*(t_0 \rightarrow \infty), \bm{u}^*(t_0 \rightarrow \infty)$, for the fulfillment of Assumption~\ref{assm:excitation_HJBident_online} for the HJB identification further excitation signals such as probing noise or a sum of sine functions can be added to $\bm{x}^*(t_0 \rightarrow \infty)$ to compute $\bm{m}_{\hjb_i}^{(r)\top}(t)$ and $z_{\hjb_i}^{(r)}(t)$.
\end{remark}

\tdd{Kürzungspotential (2): Algorithmus streichen}
\begin{algorithm}[t]
	\caption{Online HJB-Based IDG Method}
	\label{alg:onlineIDGHJB}
	\DontPrintSemicolon
	
	\KwIn{Current measurement of GT trajectories $\bm{x}^*(t)$ and $\bm{u}^*(t)$, $\bm{f}(\bm{x})$, $\bm{G}_i(\bm{x}), \forall i \in \mathcal{N}$, $\bm{\psi}_i(\bm{x}), \bm{\phi}_i(\bm{x}), \forall i \in \mathcal{N}$ and learning rates $\tau_i, \kappa_i, \forall i \in \mathcal{N}$.}
	\KwOut{Current cost function parameter estimation $\hat{\bm{\beta}}_i(t)$ and $\hat{\bm{R}}_{ij}(t), \forall i,j \in \mathcal{N}$.}
	
	\For{$i \in \mathcal{N}$}{
		Compute value function parameter update $\dot{\hat{\bar{\bm{\theta}}}}_i$~\eqref{eq:FNElearninglaw}. \;
		Compute $\hat{\bm{\mu}}_i(\bm{x})$~\eqref{eq:mu_i_hat} and $\hat{\bm{\theta}}_i^{(r)}$~\eqref{eq:theta_r_hat}. \;
		\If{\eqref{eq:stoppingCriterion} smaller than threshold}{
			Set $\tau_i = 0$.
		}
	}
	
	\For{$i \in \mathcal{N}$}{
		Compute cost function parameter update $\dot{\hat{\bm{\eta}}}_i^{(r)}$~\eqref{eq:HJB_learningLaw}. \;
		\If{\eqref{eq:stoppingCriterion} with $\hat{\bm{\eta}}_i^{(r)}$ smaller than threshold}{
			Set $\kappa_i = 0$.
		}
		\Return $\hat{\bm{\beta}}_i(t)$ and $\hat{\bm{R}}_{ij}(t), \forall j \in \mathcal{N}$.
	}
\end{algorithm}


\subsection{Approximated Value Function Structures} \label{subsec:IDGHJB_online_approx_valueFct_structures}
\tdd{similarly to offline case, give lemma for value function structure approximation error, only discuss remaining case, remark in general worse solution (higher error) than in offline approach}

In the following, we analyze the case where Assumption~\ref{assm:V_i_star} is not fulfilled and the basis functions $\bm{\phi}_i, \forall i \in \mathcal{N}$ are chosen in line with an approximation theorem. Lemma~\ref{lemma:FNEident_valueFct_error_online} states that a bounded error between the online identified FNE strategy $\hat{\bm{\mu}}^{(\text{online})}_i(\bm{x})$ and the GT one $\bm{\mu}^*_i(\bm{x}), \forall i \in \mathcal{N}$ results. Then, with Lemma~\ref{lemma:HJBident_valueFct_error_online}, we show that under Assumption~\ref{assm:Q_i_FNEident_online} the online estimated parameters $\hat{\bm{\beta}}_i$ and $\hat{\bm{R}}_{ij}, \forall i,j \in \mathcal{N}$ converge to fixed values which yield the identified FNE $\{ \hat{\bm{\mu}}^{(\text{online})}_i: i \in \mathcal{N} \}$.

\begin{lemma} \label{lemma:FNEident_valueFct_error_online}
	Let Assumptions~\ref{assm:R_ij}, \ref{assm:systemDynamics_basisFunctions_known} and \ref{assm:excitation_FNEident_online} hold. If Assumption~\ref{assm:V_i_star} is not fulfilled but the basis functions~$\bm{\phi}_i$ are as defined in Lemma~\ref{lemma:FNEident_valueFct_error_offline}, the parameter adaption scheme \eqref{eq:FNElearninglaw} exponentially converges to a residual set\comment{therefore, the specially designed stopping criterion, parameters can oscillate} such that $\norm{ \bar{\bm{\theta}}_i^{(\text{err})} } \leq \Theta_i < \infty$ ($\Theta_i \in \mathbb{R}_{\geq 0}$) and
	\begin{align} \label{eq:FNEident_online_bounded}
		\norm{ \hat{\bm{\mu}}^{(\text{online})}_i(\bm{x}) - \bm{\mu}^*_i(\bm{x}) } \leq \norm{ \tilde{\bm{\mu}}^{*}_i(\bm{x}) - \bm{\mu}^*_i(\bm{x}) } \notag \\+ \frac{1}{2}\norm{\bar{\bm{\Phi}}_i}\Theta_i < \infty,
	\end{align}
	where $\hat{\bm{\mu}}^{(\text{online})}_i(\bm{x}) = -\frac{1}{2} \bar{\bm{\Phi}}_i(\bm{x}) \hat{\bar{\bm{\theta}}}_i$ ($\hat{\bar{\bm{\theta}}}_i$ from the residual set and fixed) is the online identified FNE strategy, $\tilde{\bm{\mu}}^{*}_i(\bm{x}) = -\frac{1}{2} \bar{\bm{\Phi}}_i(\bm{x}) \bar{\bm{\theta}}^*_i$ the offline identified one and $\bar{\bm{\theta}}_i^{(\text{err})} = \hat{\bar{\bm{\theta}}}_i - \bar{\bm{\theta}}^*_i$.
\end{lemma}
\begin{proof}
	From Lemma~\ref{lemma:FNEident_valueFct_error_offline}, the existence of the bounded error function $\bar{\varepsilon}_i(\bm{x})$ and parameters $\bm{\theta}_i^*$ follows such that $\forall \bm{x} \in \mathcal{X}$ $V_i^*(\bm{x}) = {\bm{\theta}_i^*}^\top \bm{\phi}_i(\bm{x}) + \bar{\varepsilon}_i(\bm{x})$. The parameters $\bm{\theta}_i^*$ define the offline identified FNE strategy $\tilde{\bm{\mu}}^{*}_i(\bm{x}) = -\frac{1}{2} \bar{\bm{\Phi}}_i(\bm{x}) \bar{\bm{\theta}}^*_i$ according to Lemma~\ref{lemma:FNEident_valueFct_error_offline}.
	
	Now, by defining $\bar{\bm{\theta}}^{(\text{err})}_i(t) = \hat{\bar{\bm{\theta}}}_i(t) - \bar{\bm{\theta}}_i^*$ as the error between the online and offline identified FNE parameters, we get the error dynamics
	\begin{align} \label{eq:proof_valueFct_error_online_1}
		\dot{\bar{\bm{\theta}}}^{(\text{err})}_i&(t) = - \tau_i \bm{M}^\top_{u_i}(t) \bm{M}_{u_i}(t) \bar{\bm{\theta}}^{(\text{err})}_i(t) \notag \\
		&- \tau_i\bm{M}^\top_{u_i}(t) \left( \tilde{\bm{\mu}}^{*}_i(\bm{x}^*(t)) - \bm{\mu}^*_i(\bm{x}^*(t)) \right).
	\end{align}
	According to \cite[Technical Lemma 2]{Vamvoudakis.2010}, since $\bm{M}_{u_i}^\top(t)$ is PE, $\norm{\tilde{\bm{\mu}}^{*}_i(\bm{x}^*(t)) - \bm{\mu}^*_i(\bm{x}^*(t))}<\infty$ (cf.~Lemma~\ref{lemma:FNEident_valueFct_error_offline}) and $\bm{M}_{u_i}^\top(t)$ and $\bar{\bm{\theta}}^{(\text{err})}_i(t)$\comment{ESL with bounded error} are bounded, we find $\Theta_i \in \mathbb{R}_{\geq 0}$ such that $\hat{\bar{\bm{\theta}}}_i$ exponentially converges to the set~$\norm{ \bar{\bm{\theta}}_i^{(\text{err})} } \leq \Theta_i < \infty$.
	
	With a fixed value $\hat{\bar{\bm{\theta}}}_i$ from this residual set, we can establish
	\begin{align} \label{eq:proof_valueFct_error_online_2}
		&\norm{ \hat{\bm{\mu}}^{(\text{online})}_i(\bm{x}) - \bm{\mu}^*_i(\bm{x}) } = \notag \\
		&\norm{ -\frac{1}{2} \bar{\bm{\Phi}}_i(\bm{x}) \bar{\bm{\theta}}^*_i - \bm{\mu}^*_i(\bm{x}) -\frac{1}{2} \bar{\bm{\Phi}}_i(\bm{x}) \bar{\bm{\theta}}^{(\text{err})}_i }
	\end{align}
	and thus, the upper bound \eqref{eq:FNEident_online_bounded}.
\end{proof}

\begin{remark} \label{remark:greaterError_onlineFNEident}
	In general, the error of the online identified FNE strategy $\hat{\bm{\mu}}^{(\text{online})}_i(\bm{x})$ to the GT FNE strategy $\bm{\mu}^*_i(\bm{x})$ is greater than with the offline identified one $\tilde{\bm{\mu}}^{*}_i(\bm{x})$ when considering arbitrary $\bm{x} \in \mathcal{X}$. This is due to the bounded error between the online identified parameters $\hat{\bar{\bm{\theta}}}_i$ from the residual set and the offline determined parameters $\bar{\bm{\theta}}^*_i$ which are by its computation (cf.~Lemma~\ref{lemma:FNEident_valueFct_error_offline}) the best parameters approximating $\bm{\mu}^*_i(\bm{x})$ on $\mathcal{X}$ with the chosen basis functions~$\bm{\phi}_i$.
\end{remark}

\begin{assumption} \label{assm:Q_i_FNEident_online}
	Parameters $\tilde{\bm{R}}_{ij}$, $\tilde{\bm{\beta}}_i, \forall i,j \in \mathcal{N}$ exist such that they define with $\tilde{Q}_i(\bm{x}) = \tilde{\bm{\beta}}_i^\top \bm{\psi}_i(\bm{x})$ a DG as given in Section~\ref{sec:problem} with the unique FNE $\{\hat{\bm{\mu}}^{(\text{online})}_i: i \in \mathcal{N}\}$ and its corresponding value functions $\tilde{V}^*_i(\bm{x}) = {\hat{\bm{\theta}}_i}^\top \bm{\phi}_i(\bm{x})$.
\end{assumption}

\begin{lemma} \label{lemma:HJBident_valueFct_error_online}
	Let Assumptions~\ref{assm:R_ij}, \ref{assm:systemDynamics_basisFunctions_known}, \ref{assm:excitation_HJB_offline}, \ref{assm:excitation_FNEident_online}, \ref{assm:excitation_HJBident_online} and \ref{assm:Q_i_FNEident_online} hold $\forall i \in \mathcal{N}$. Furthermore, let the basis functions~$\bm{\phi}_i, \forall i \in \mathcal{N}$ be chosen according to Lemma~\ref{lemma:FNEident_valueFct_error_offline} and let the identification of the FNE strategy $\tilde{\bm{\mu}}^{*}_i(\bm{x}), \forall i \in \mathcal{N}$ be performed by Lemma~\ref{lemma:FNEident_valueFct_error_online}.
	Then, for every player $i \in \mathcal{N}$ the parameter adaption scheme \eqref{eq:HJB_learningLaw} yields convergence to a fixed element of the solution set \eqref{eq:HJB_ident_offline_solSet_1}, which contains all parameters that yield a FNE~$\{\hat{\bm{\mu}}^*_i: i \in \mathcal{N}\}$ with $\hat{\bm{\mu}}_i^*(\bm{x}) = \hat{\bm{\mu}}^{(\text{online})}_i(\bm{x}), \forall i \in \mathcal{N}$. Thus, for $\hat{\bm{\mu}}_i^*(\bm{x}), \forall i \in \mathcal{N}$ the upper bound~\eqref{eq:FNEident_online_bounded} holds as well.
\end{lemma}
\begin{proof}
	According to Lemma~\ref{lemma:FNEident_valueFct_error_online}, the parameter adaption scheme~\eqref{eq:FNElearninglaw} converges for all players $i \in \mathcal{N}$ exponentially fast to the residual set $\norm{\bar{\bm{\theta}}^{(\text{err})}_i} \leq \Theta < \infty, \forall i \in \mathcal{N}$. Due to the stopping criterion~\eqref{eq:stoppingCriterion}, the parameter adaption is stopped with a fixed value $\hat{\bar{\bm{\theta}}}_i$ from this residual set, which yields $\hat{\bm{\mu}}^{(\text{online})}_i(\bm{x})$ with the upper bound~\eqref{eq:FNEident_online_bounded}.
	
	If the combination of the identified control laws $\hat{\bm{\mu}}^{(\text{online})}_i(\bm{x})$ constitute a FNE of a DG as defined in Section~\ref{sec:problem} with cost functions $\tilde{Q}_i(\bm{x}) = \tilde{\bm{\beta}}_i^\top \bm{\psi}_i(\bm{x})$ with unknown parameters~$\tilde{\bm{\beta}}_i$ (cf.~Assumption~\ref{assm:Q_i_FNEident_online}), the solution sets~\eqref{eq:HJB_ident_offline_solSet_1} define the set of all parameters $\hat{\bm{\alpha}}_i, \hat{\bm{\beta}}_i, \forall i \in \mathcal{N}$ that lead to a FNE~$\{\hat{\bm{\mu}}^*_i: i \in \mathcal{N}\}$ with $\hat{\bm{\mu}}_i^*(\bm{x}) = \hat{\bm{\mu}}^{(\text{online})}_i(\bm{x}), \forall i \in \mathcal{N}$ (cf.~Theorem~\ref{theorem:HJB_ident_offline}). Finally, the convergence of the parameter adaption schemes~\eqref{eq:HJB_learningLaw} of all players $i \in \mathcal{N}$ to a fixed element of the sets~\eqref{eq:HJB_ident_offline_solSet_1} follows from the proof of Theorem~\ref{theorem:HJBident_online}.
\end{proof}


\subsection{Approximated Cost Function Structures} \label{subsec:IDGHJB_online_approx_costFct_structures}
As for our new offline IDG method, it is again important to know the cost function structures for a given/identified FNE (see Assumption~\ref{assm:Q_i} and \ref{assm:Q_i_FNEident_online}). 

If an approximation structure is chosen such that $\varepsilon_i(\bm{x})$ is not approximately zero, a bounded perturbation in the error dynamics~\eqref{eq:proof_HJBident_online_1}\comment{error to the offline computed parameters which define, as for the FNE identification, the best possible approximation with the used basis functions~$\bm{\psi}_i$} of the parameter adaption scheme~\eqref{eq:HJB_learningLaw} of the cost function parameters~$\hat{\bm{\eta}}_i^{(r)}$ results. Since the regression vector~$\bm{m}_{\hjb_i}^{(r)}(t)$ is in most cases only PE in a subspace and not in the complete parameter space, exponential stability of the origin of the unperturbed error dynamics cannot be guaranteed. Now, due to the additional bounded perturbation in \eqref{eq:proof_HJBident_online_1} from the bounded error function~$\varepsilon_i(\bm{x})$, it cannot be guaranteed that $\bm{\eta}_i^{(r,\text{err})}$ and thus, $\hat{\bm{\eta}}_i^{(r)}$ is bounded \cite{Jenkins.2018}\comment{remarkably, the learning process can still result in bounded errors but no general guarantees can be given, see numerical example no unstable learning process but very high error bounds when sufficiently excited, see excitation with sum of sine functions}. Hence, in general, the learning process can be unstable. The missing PE of $\bm{m}_{\hjb_i}^{(r)}(t)$ in the complete parameter space follows from the characteristic non-uniqueness of IDG solutions. Full PE and thus, bounded errors~$\bm{\eta}_i^{(r,\text{err})}$ in case of a bounded error functions~$\varepsilon_i(\bm{x})$ can only be guaranteed when the IDG solution is unique up to a scaling factor, i.e. the solution sets~\eqref{eq:HJB_ident_offline_solSet_1} consist of only one element (cf.~Theorem~\ref{theorem:HJB_ident_offline}).




\section{Numerical Example} \label{sec:example}
We define a two-player example system with
\begin{align}
	\bm{G}(\bm{x}) \! &= \!\!\mat{ \bm{G}_1(\bm{x}) \! \! & \!\! \bm{G}_2(\bm{x}) } = \mat{ 0 & 0 \\ \cos(2x_1)\!+\!2 & \sin(4x_1^2)\!+\!2 } \label{eq:example_system_G}, \\
	\bm{f}(\bm{x})\! &= \!\!\mat{ -2x_1+x_2 \\ \!-x_2\!\!-\!\frac{1}{2}x_1\!\!+\!\frac{1}{4}x_2( ( \cos(2x_1)\!+\!2 )^2 \!\!+\! ( \sin(4x_1^2)\!+\!2 )^2 ) } \label{eq:example_system_f}
\end{align}
and the value functions
\begin{align} \label{eq:example_system_VFs}
	V_1^{\bm{\mu}}(\bm{x}) &= \int_{t}^{\infty} 2( x_1^2 + x_2^2 ) + 2( \mu_1^2 + \mu_2^2 ) \text{d}\tau,\\
	V_2^{\bm{\mu}}(\bm{x}) &= \frac{1}{2} V_1^{\bm{\mu}}(\bm{x}).
\end{align}
According to \cite{Karg.2023} and \cite{Nevistic.1996}, the optimal value functions are given by $V_1^*(\bm{x})=\frac{1}{2}x_1^2+x_2^2$ and $V_2^*(\bm{x})=\frac{1}{4}x_1^2+\frac{1}{2}x_2^2$. Thus, we define $\bm{\phi}_1(\bm{x})=\bm{\phi}_2(\bm{x})=\mat{ x_1^2 & x_1x_2 & x_2^2 }^\top$ and $\bm{\psi}_1(\bm{x})=\bm{\psi}_2(\bm{x})=\mat{ x_1^2 & x_1x_2 & x_2^2 }^\top$ and fulfill with $\bm{\theta}_1^* = \mat{\frac{1}{2} & 0 & 1}^\top$, $\bm{\theta}_2^*=\frac{1}{2}\bm{\theta}_1^*$, $\bm{\beta}_1^*=\mat{ 2 & 0 & 2 }^\top$ and $\bm{\beta}_2^*=\frac{1}{2}\bm{\beta}_1^*$ Assumptions~\ref{assm:Q_i} and \ref{assm:V_i_star}. Furthermore, $\bm{\alpha}_1^*=\mat{ 2 & 2 }^\top$ and $\bm{\alpha}_2^*=\frac{1}{2}\bm{\alpha}_1^*$. The GT trajectories $\bm{x}^*(t_0\rightarrow \infty), \bm{u}^*(t_0 \rightarrow \infty)$ result from applying the GT FNE $\bm{\mu}^*(\bm{x})$, which is computed by $V_1^*(\bm{x})$ and $V_2^*(\bm{x})$, to the system defined by \eqref{eq:example_system_f}, \eqref{eq:example_system_G} where the system state is successively reset to the following initial states every $\bigtriangleup t_\text{init} = 2\,\text{s}$ ($j\in\{0,1,\dots,7\}$) starting with $t_0=0\,\text{s}$\footnote{The simulations are performed in the Matlab environment with an Intel Core i7-8550U.}: $\mat{3&\!\!1}^\top$, $\mat{1&\!\!3}^\top$, $\mat{-3&\!\!1}^\top$, $\mat{-1&\!\!3}^\top$, $\mat{3&\!\!-1}^\top$, $\mat{1&\!\!-3}^\top$, $\mat{-3&\!\!-1}^\top$, $\mat{-1&\!\!-3}^\top$.


\subsection{Known Cost and Value Function Structures} \label{subsec:example_known_structures}
\tdd{illustrate result first and second theorem, i.e. main results paper = solutions to first and second problem - keep explanations very short, should all be clear from the theory}

\begin{figure}[t]
	\centering
	\begin{subfigure}[t]{1.65in}
		\includegraphics[width=1.65in]{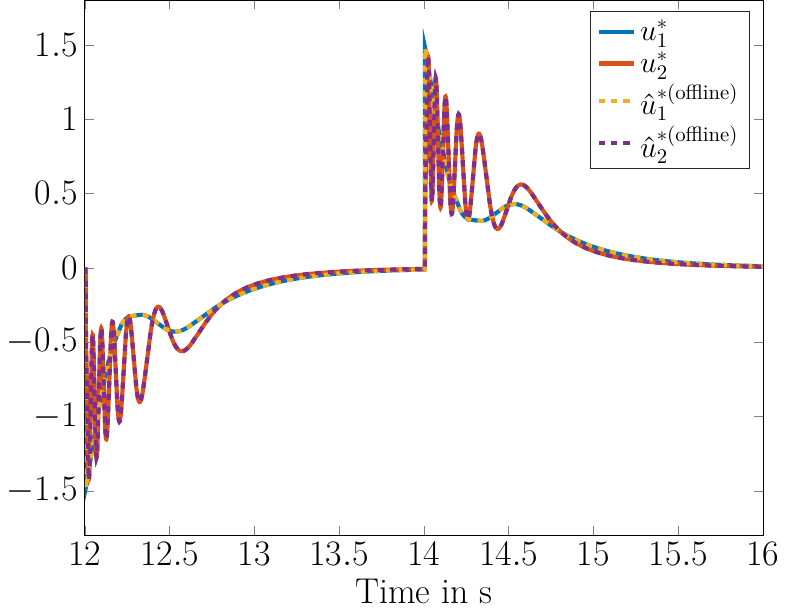}
	\end{subfigure}
	\begin{subfigure}[t]{1.65in}
		\includegraphics[width=1.65in]{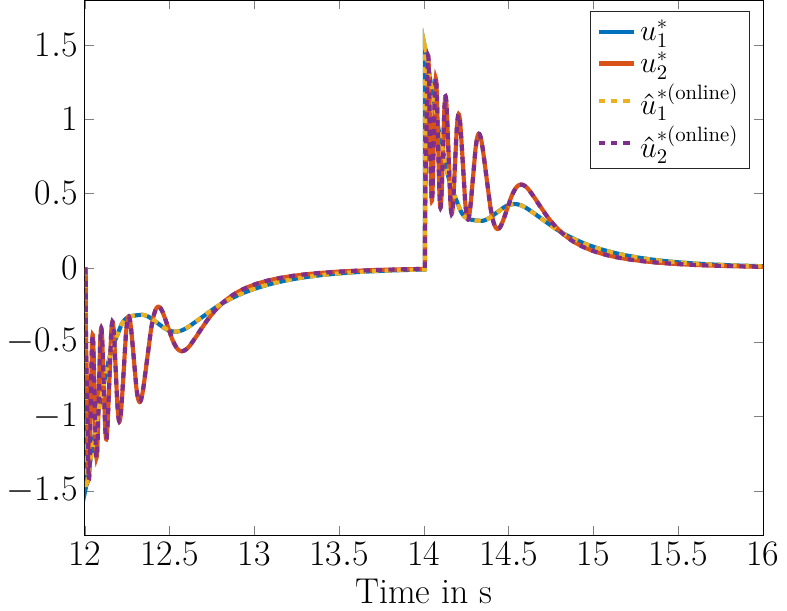}
	\end{subfigure}
	\caption{Comparison of GT trajectories (solid lines) and trajectories (dashed) resulting from the parameters determined by the offline (left) and converged online (right) IDG method for the last two initial state resets.}
	\label{fig:comparisonGTvsIDGErrorFree}
\end{figure}

\begin{figure}[t]
	\centering
	\begin{subfigure}[t]{1.65in}
		\includegraphics[width=1.65in]{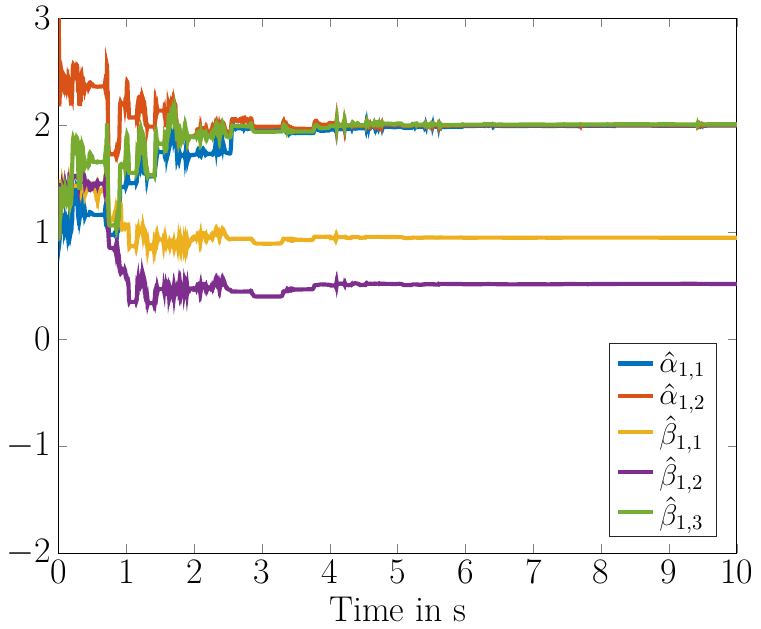}
	\end{subfigure}
	\begin{subfigure}[t]{1.65in}
		\includegraphics[width=1.65in]{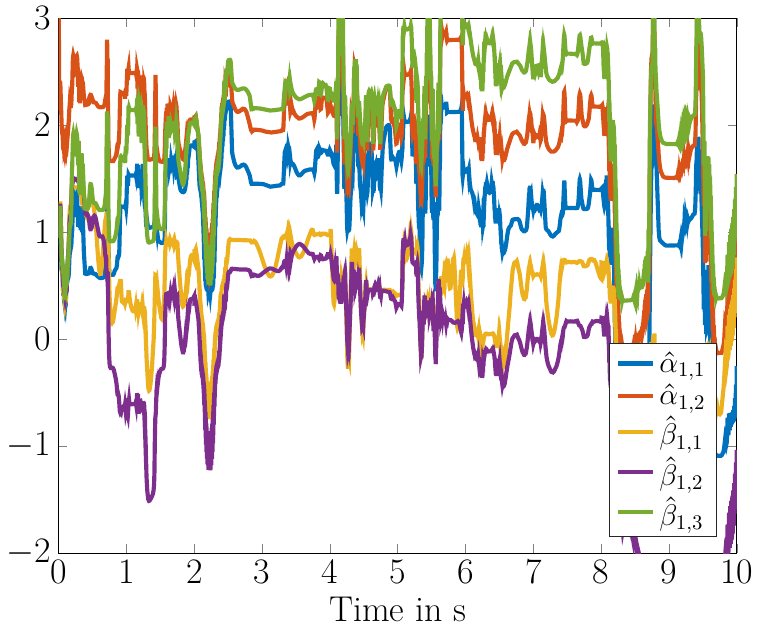}
	\end{subfigure}
	\caption{Parameter trajectories of the online IDG method in case Assumption~\ref{assm:Q_i} is fulfilled (left) and violated (right).}
	\label{fig:learningTrajPlayer1}
\end{figure}

By evaluating the basis functions $\bm{\phi}_1$ and $\bm{\phi}_2$ according to \eqref{eq:mu_i_red}, we define the reduced basis function vectors $\bm{\phi}_1^{(r)}(\bm{x}) = \bm{\phi}_2^{(r)}(\bm{x}) = \mat{ x_1x_2 & x_2^2 }^\top$. Setting up the data matrices $\bm{M}_{u_i}$ and $\bm{z}_{u_i}$ from the GT data with $\bigtriangleup t = 0.001\,\text{s}$ reveals that Assumption~\ref{assm:excitation_FNE_offline} is fulfilled with $\rank(\bm{M}_{u_i})=p_i\bar{h}_i, \forall i \in \{1,2\}$. Hence, the FNE identification according to Lemma~\ref{lemma:FNE_ident_offline} is unique and we can use the solution set~\eqref{eq:HJB_ident_offline_solSet_1} to solve Problem~\ref{problem:offline_idg}. To set up $\bm{M}_{\hjb_i}^{(r)}$ and $\bm{z}_{\hjb_i}^{(r)}, \forall i \in \{1,2\}$ to fulfill Assumption~\ref{assm:excitation_HJB_offline}, we use an equidistant two-dimensional grid around the origin with grid size $1$ and with upper and lower bounds at $10$ and $-10$, respectively.
The unique FNE identification yields $\hat{\bm{\theta}}_1^{(r)} = \hat{\bm{\theta}}_2^{(r)} \approx \mat{0 & 1}^\top$ and we determine one element of the solution set~\eqref{eq:HJB_ident_offline_solSet_1} with $\hat{\bm{\alpha}}_1 = \hat{\bm{\alpha}}_2 \approx \mat{ 2 & 2}^\top$, $\hat{\bm{\beta}}_1 \approx \mat{1.004 & 0.498 & 2}^\top$ and $\hat{\bm{\beta}}_2 \approx \mat{ 1.017 & 0.492 & 2 }^\top$. Here, in \eqref{eq:HJB_ident_offline_solSet_1}, $\left(\bm{I} - \bar{\bm{M}}^{(r)^+}_{\hjb_i} \bar{\bm{M}}^{(r)}_{\hjb_i} \right) \bar{\bm{w}}^{(r)}_i$ reduces to $\bar{w}^{(r)}_i \mat{ 0 & 0 & 0.873 & -0.436 & 0 & 0.218 }^\top, \forall i \in \{1,2\}$ with $\bar{w}^{(r)}_i \in \mathbb{R}$ as arbitrary scalar. For the given estimates, $\bar{w}^{(r)}_1 = 0.714$ and $\bar{w}^{(r)}_2 = 0.729$ hold. Remarkably, the solution set~\eqref{eq:HJB_ident_offline_solSet_1} of Problem~\ref{problem:offline_idg} is now completely defined by every $\bar{w}^{(r)}_i$ that yields $\hat{\bm{\beta}}_i$ such that $Q_i$ is positive definite. Solving the coupled HJB equations with the estimated cost function parameters $\hat{\bm{\alpha}}_i, \hat{\bm{\beta}}_i, \forall i \in \{1,2\}$ by the PI algorithm described in \cite{Karg.2023} leads to the FNE $\hat{\bm{\mu}}^*(\bm{x})$ and the trajectories $\hat{\bm{x}}^*(t_0 \rightarrow \infty)$ and $\hat{\bm{u}}^*(t_0 \rightarrow \infty)$. Fig.~\ref{fig:comparisonGTvsIDGErrorFree} (left subfigure) illustrates that $\hat{\bm{u}}^*(t_0 \rightarrow \infty) \approx \bm{u}^*(t_0 \rightarrow \infty)$ and thus, also $\hat{\bm{x}}^*(t_0 \rightarrow \infty) \approx \bm{x}^*(t_0 \rightarrow \infty)$\comment{unique solution system dynamics}. This can quantitatively verified by defining the normalized sum of absolute errors (NSAE) for the system state~$\delta^x$ and the control signals~$\delta^u$: 
\begin{align} 
	\delta^x = \sum_{j=1}^{n} \frac{1}{\max_{k}\abs{x_j^{*(k)}}} \sum_{k=0}^{K-1} \abs{ \hat{x}_j^{*(k)} - x_j^{*(k)} } \label{eq:nsae_x}, \\
	\delta^u = \sum_{j=1}^{p} \frac{1}{\max_{k}\abs{u_j^{*(k)}}} \sum_{k=0}^{K-1} \abs{ \hat{u}_j^{*(k)} - u_j^{*(k)} } \label{eq:nsae_u},
\end{align}
where the index $k$ denotes the data point at $t_0+k\bigtriangleup t$. With the offline IDG method, $\delta^x \approx 0.010$ and $\delta^u \approx 0.011$ are achieved which shows that Problem~\ref{problem:offline_idg} is solved. Lastly, for completeness, the overall computation of $\hat{\bm{\alpha}}_i, \hat{\bm{\beta}}_i, \forall i \in \{1,2\}$ from the GT data took only $2.22\,\text{s}$.

Regarding our new online IDG method, with the GT trajectories Assumption~\ref{assm:excitation_FNEident_online} is fulfilled. Hence, the online identification of the FNE strategies of both players converges exponentially to the GT ones. Following Remark~\ref{remark:excitation_FNEvsHJB}, we use a sum of sine functions to compute $\bm{m}_{\hjb_i}^{(r)\top}(t)$ and $z_{\hjb_i}^{(r)}(t)$ to fulfill Assumption~\ref{assm:excitation_HJBident_online}. Precisely, we replace $\bm{x}^*(t_0 \rightarrow \infty)$ with $\bm{x}^\top_{\hjb}(t) = \mat{ \sum_{i=1}^{3} 3\sin(2\pi f_i^{(1)} t) & \sum_{i=1}^{3} 3\sin(2\pi f_i^{(2)} t) }$, where $f_i^{(1)}$ and $f_i^{(2)}$ are randomly chosen for every system state reset from a uniform distribution between $0.5\,\text{Hz}$ and $5\,\text{Hz}$. Finally, in Fig.~\ref{fig:learningTrajPlayer1} (left subfigure) the trajectories of the estimated cost function parameters $\hat{\bm{\alpha}}_1, \hat{\bm{\beta}}_1$ are exemplary shown for the first player. For both players, the parameter estimates converge after ca. $6\,\text{s}$ to $\hat{\bm{\alpha}}_1 \approx \mat{ 1.998 & 1.999 }^\top$, $\hat{\bm{\alpha}}_2 \approx \mat{ 2.001 & 1.998 }^\top$, $\hat{\bm{\beta}}_1 \approx \mat{ 0.948 & 0.514 & 2.005 }^\top$ and $\hat{\bm{\beta}}_2 \approx \mat{ 0.949 & 0.516 & 2.001 }^\top$. Thus, we find with $\bar{w}_1^{(r)}=\bar{w}_2^{(r)}\approx0.65$ scalars such that the estimates are part of the solution set \eqref{eq:HJB_ident_offline_solSet_1}. This illustrates Theorem~\ref{theorem:HJBident_online}. Furthermore, the converged parameters solve Problem~\ref{problem:online_idg}, which is qualitatively shown in the right subfigure of Fig.~\ref{fig:comparisonGTvsIDGErrorFree} by the match of GT and estimated trajectories and quantitatively by small NSAE values: $\delta^x \approx 4.678$ and $\delta^u \approx 4.668$. 


\subsection{Approximated Value Function Structures} \label{subsec:example_approx_valueFct_structures}

\comment{DG with chosen basis functions~$\bm{\psi}_i$ (noticeably, unique DG parameters in case where $x_1 x_2$ is dropped, i.e. $x_1 x_2$ introduces one additional DOF) exists for identified FNE, i.e. Assumption~\ref{assm:Q_i_FNEident} and \ref{assm:Q_i_FNEident_online} are fulfilled, if identified FNE results from value function basis functions without $x_1 x_2$ (i.e. second element needs to be identified with coefficient zero or dropped - last approach here) and identified relations $\frac{\hat{\theta}_{12}}{\hat{\alpha}_{11}}$ resp. $\frac{\hat{\theta}_{22}}{\hat{\alpha}_{22}}$ is $\approx \frac{1}{2}$, design VF error accordingly (needs to be checked with $\hat{\alpha}_{11}$ resp. $\hat{\alpha}_{22}$, since due to the scaling ambiguity $\hat{\theta}_{12}$ and $\hat{\theta}_{22}$ are set to one always); otherwise, it would be possible based on Remark~\ref{remark:choice_basisFunction_costFunction} to introduce new basis functions: $x_2^2(\cos(2x_1)+2)^2$ in $\bm{\psi}_1$ and $x_2^2(\sin(4x_1^2)+2)^2$ in $\bm{\psi}_2$}

\begin{figure}[t]
	\centering
	\begin{subfigure}[t]{1.65in}
		\includegraphics[width=1.65in]{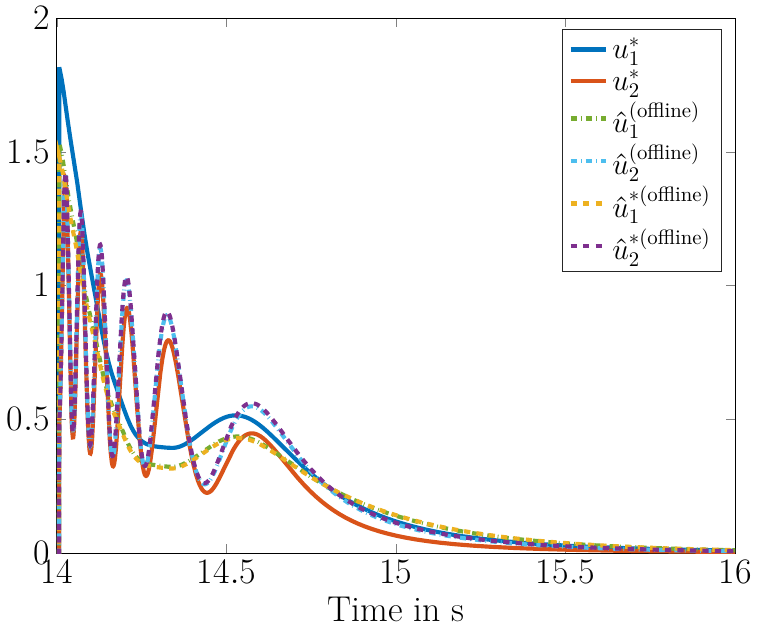}
	\end{subfigure}
	\begin{subfigure}[t]{1.65in}
		\includegraphics[width=1.65in]{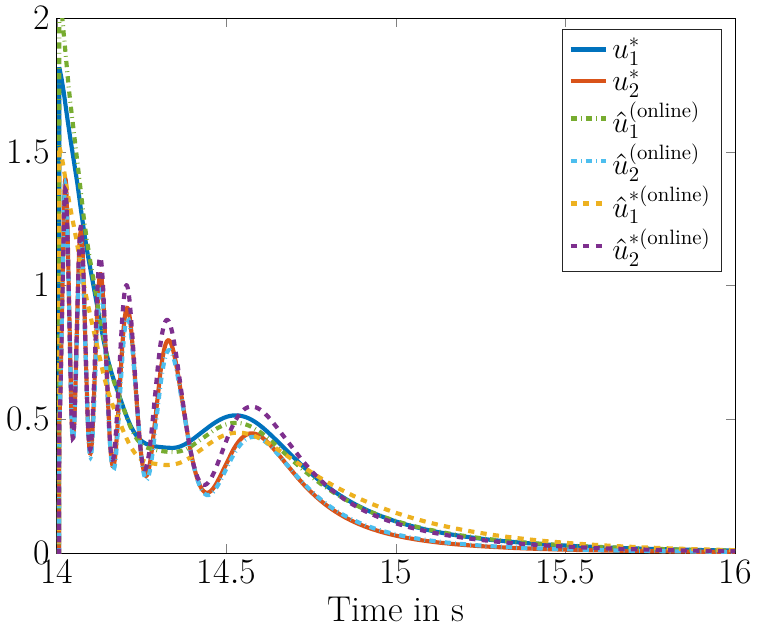}
	\end{subfigure}
	\caption{Comparison of GT trajectories (solid lines), trajectories (dash-dotted) resulting from the FNE identification and trajectories (dashed) resulting from parameters determined by the IDG method for the last initial state reset. Left subfigure shows result of offline method and right one for converged online approach where Assumption~\ref{assm:V_i_star} is violated.}
	\label{fig:comparisonGTvsFNEvsIDGVFerror}
\end{figure}

\begin{figure}[t]
	\centering
	\begin{subfigure}[t]{1.65in}
		\includegraphics[width=1.65in]{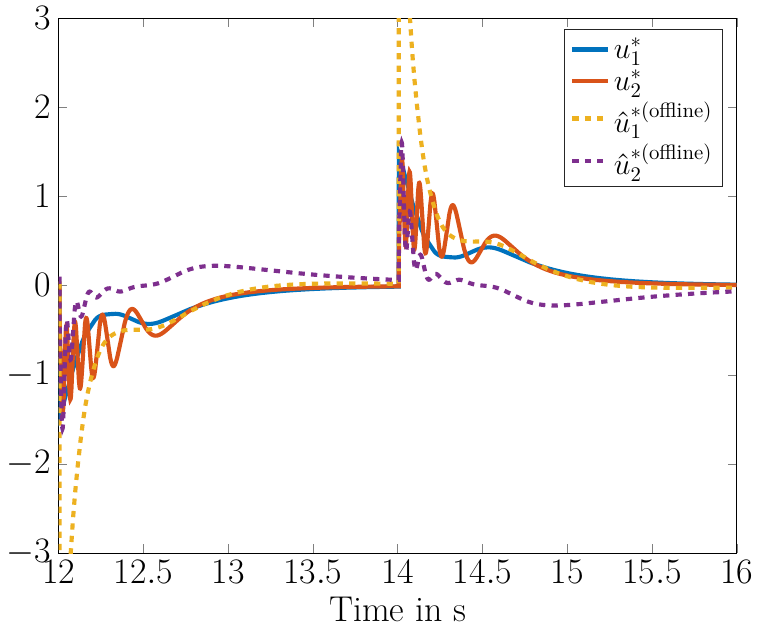}
	\end{subfigure}
	\begin{subfigure}[t]{1.65in}
		\includegraphics[width=1.65in]{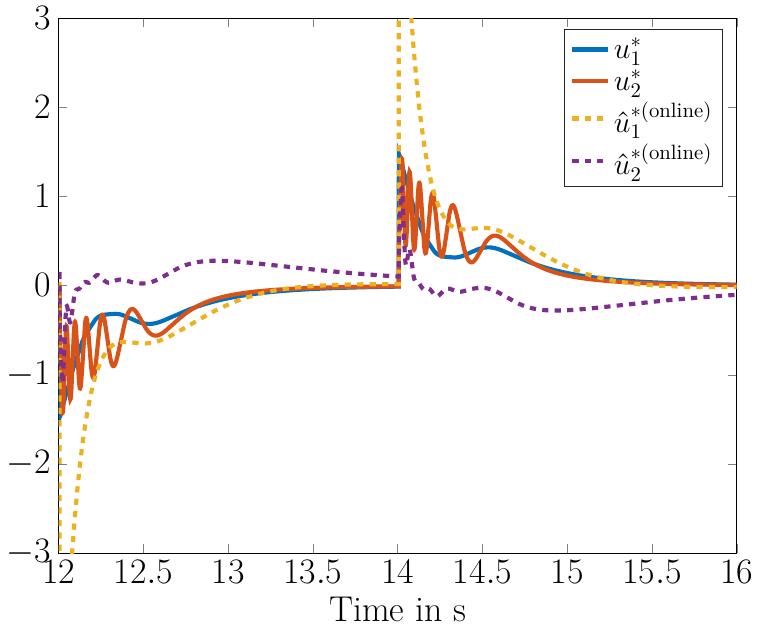}
	\end{subfigure}
	\caption{Comparison of GT trajectories (solid lines) and trajectories (dashed) resulting from the parameters determined by the offline (left) and online (right) IDG method for the last two initial state resets where Assumption~\ref{assm:Q_i} is violated.}
	\label{fig:comparisonGTvsIDGQerror}
\end{figure}

We simulate the case of not fulfilled Assumption~\ref{assm:V_i_star} by computing the GT FNE strategies $\mu_1^*(\bm{x})$ and $\mu_2^*(\bm{x})$ with $V_1^*(\bm{x}) = \bm{\theta}_1^{*\top}\bm{\phi}_1(\bm{x}) + \frac{1}{2}\sin(x_2^2)$ and $V_2^*(\bm{x})=\bm{\theta}_2^{*\top}\bm{\phi}_2(\bm{x})$, respectively.
Regarding the offline IDG method, according to Lemma~\ref{lemma:FNEident_valueFct_error_offline}, a bounded error between the identified control laws~$\hat{\mu}_i(\bm{x}), \forall i \in \{1,2\}$ and the FNE strategies $\mu^*_i(\bm{x})$ is guaranteed. When using $\bm{\phi}_1(\bm{x})=\bm{\phi}_2(\bm{x})=\mat{ x_1^2 & x_2^2 }^\top$ as basis functions in the FNE identification step, it can be verified that Assumption~\ref{assm:Q_i_FNEident} is fulfilled with the identified parameters $\hat{\bar{\theta}}_1 = \hat{\bar{\theta}}_2 \approx \frac{1}{2}$ since parameters $\bm{\beta}_i,\forall i \in \{1,2\}$ exist such that the coupled HJB equations can be fulfilled $\forall \bm{x} \in \mathcal{X}$. Thus, finally, the estimated parameters $\hat{\bm{\alpha}}_1 = \hat{\bm{\alpha}}_2 \approx \mat{1.998 & 2.000}^\top$ and $\hat{\bm{\beta}}_1 = \hat{\bm{\beta}}_2 \approx \mat{ 1.004 & 0.498 & 2.000 }^\top$ yield a FNE~$\hat{\bm{\mu}}^*(\bm{x})$ matching the identified control laws, i.e. $\hat{\mu}^*_i(\bm{x}) = \hat{\mu}_i(\bm{x}), \forall i \in \{1,2\}$. Fig.~\ref{fig:comparisonGTvsFNEvsIDGVFerror} (left subfigure) illustrates this: The trajectories $\hat{u}^*_i(t_0\rightarrow\infty), \forall i \in \{1,2\}$ resulting from the IDG solution coincide with the trajectories $\hat{u}_i(t_0\rightarrow\infty), \forall i \in \{1,2\}$ resulting from the identified FNE strategies (cf.~Lemma~\ref{lemma:HJBident_valueFct_error_offline}). Both trajectories show bounded errors to the GT trajectories $u^*_i(t_0\rightarrow\infty), \forall i \in \{1,2\}$. Quantitatively, this can be verified by similar NSAE values $\delta^x \approx 404.8$ and $\delta^u \approx 920.3$ for $\hat{\bm{x}}(t_0\rightarrow\infty), \hat{\bm{u}}(t_0\rightarrow\infty)$ and $\delta^x \approx 453.1$ and $\delta^u \approx 961.4$ for $\hat{\bm{x}}^*(t_0\rightarrow\infty), \hat{\bm{u}}^*(t_0\rightarrow\infty)$, which are however comprehensibly higher than in the error-free case (see Subsection~\ref{subsec:example_known_structures}). The overall computation time of the IDG method was $4.62\,\text{s}$.

Due to the bounded error function in $V_1^*$, only convergence to a residual set of the online identification of $\mu_1^*(\bm{x})$ occurs (cf.~Lemma~\ref{lemma:FNEident_valueFct_error_online}). Thus, after terminating the parameter adaption~\eqref{eq:FNElearninglaw} of $\hat{\bar{\theta}}_1$ when convergence to this residual set occurs (see stopping criterion~\eqref{eq:stoppingCriterion}), we get $\hat{\bar{\theta}}_1 \approx 0.7$, which differs from the offline result. The parameter adaption of $\hat{\bar{\theta}}_2$ still converges to the unique optimal value $\hat{\bar{\theta}}_2 \approx 0.5$ (no error function in $V_2^*$). Now, the fixed values $\hat{\bar{\theta}}_i,\forall i \in \{1,2\}$ are used to define the online identified control laws~$\hat{\mu}^{(\text{online})}_i(\bm{x}),\forall i \in \{1,2\}$ that yield higher NSAE values to the GT trajectories than the offline identified ones (cf.~Remark~\ref{remark:greaterError_onlineFNEident}): $\delta^x \approx 631.2$ and $\delta^u \approx 1321.8$. By analyzing the HJB equations, we observe that they cannot be solved $\forall \bm{x} \in \mathcal{X}$ if $\hat{\bar{\theta}}_1 \approx 0.7$. Hence, Assumption~\ref{assm:Q_i_FNEident_online} is not fulfilled in this case and we cannot apply Lemma~\ref{lemma:HJBident_valueFct_error_online} to conclude $\hat{\mu}^*_i(\bm{x})=\hat{\mu}^{(\text{online})}_i(\bm{x}), \forall i \in \{1,2\}$ and thus, $\hat{u}^*_i(t_0\rightarrow\infty)=\hat{u}^{(\text{online})}_i(t_0\rightarrow\infty), \forall i \in \{1,2\}$ for the online IDG solution. However, the FNE resulting from the offline IDG solution still is the FNE yielding the smallest possible HJB error\comment{comparable result to MP-based IDG method, basis functions $\bm{\psi}_i$ here "best possible" ones resp. sufficiently good}. The online estimated parameters $\hat{\bm{\alpha}}_1 = \hat{\bm{\alpha}}_2 \approx \mat{1.856 & 1.998}^\top$ and $\hat{\bm{\beta}}_1 = \hat{\bm{\beta}}_2 \approx \mat{ 0.967 & 0.555 & 1.939 }^\top$ lead to a FNE approximately matching the result of the offline IDG method. This is qualitatively shown in Fig.~\ref{fig:comparisonGTvsFNEvsIDGVFerror} by $\hat{u}^{*(\text{online})}_i(t_0\rightarrow\infty)\neq\hat{u}^{(\text{online})}_i(t_0\rightarrow\infty), \forall i \in \{1,2\}$ but $\hat{u}^{*(\text{online})}_i(t_0\rightarrow\infty)\approx\hat{u}^{*(\text{offline})}_i(t_0\rightarrow\infty), \forall i \in \{1,2\}$ and quantitatively by similar NSAE values $\delta^x \approx 477.6$ and $\delta^u \approx 961.5$ for $\hat{\bm{x}}^{*(\text{online})}(t_0\rightarrow\infty), \hat{\bm{u}}^{*(\text{online})}(t_0\rightarrow\infty)$ as for $\hat{\bm{x}}^{*(\text{offline})}(t_0\rightarrow\infty), \hat{\bm{u}}^{*(\text{offline})}(t_0\rightarrow\infty)$.


\subsection{Approximated Cost Function Structures} \label{subsec:example_approx_costFct_structures}
\tdd{(only) illustrate statements regarding error in cost function structure approximation, influence on all players possible, crucial to know the cost function structure for the assumed value function structure resp. given/identified FNE - offline and online case - keep explanations very short, should all be clear from the theory}

\comment{here, we only compute with the chosen basis functions~$\bm{\phi}_i$ an approximation of the FNE~$\{\hat{\bm{\mu}}_i^*: i \in \mathcal{N}\}$; hence, the part of Lemma~\ref{lemma:costFct_error_offline} that assumes that the value functions~$\hat{V}_i^*$ corresponding to the FNE~$\{\hat{\bm{\mu}}_i^*: i \in \mathcal{N}\}$ can be represented by linear combinations of $\bm{\phi}_i$ is not fulfilled; however, the remaining parts hold as given, especially note that for the conclusion that $\{\bm{\mu}_i^*: i \in \mathcal{N}\}$ is not a fix point of the PI iteration and thus, not the searched FNE, the LE with $\{\bm{\mu}_i^*: i \in \mathcal{N}\}$ has not to be solved exactly}

We investigate the case of cost function approximation errors, i.e. not fulfilled Assumption~\ref{assm:Q_i} (or Assumptions~\ref{assm:Q_i_FNEident} and \ref{assm:Q_i_FNEident_online}) by using $Q_1(\bm{x}) = \bm{\beta}_1^\top\bm{\psi}_1(\bm{x}) + \frac{1}{5}x_1^4$ to set up the HJB equation~\eqref{eq:hjb_linear_1} for the identification of the cost function parameters $\bm{\alpha}^*_1,\bm{\beta}^*_1$ of player 1. Noticeably, the FNE identification is not influenced by this and still uniquely yields the GT FNE.
Regarding the offline approach, the HJB equation of player 1 cannot be solved exactly anymore and an error according to \eqref{eq:hjb_error} remains. This leads to a bias in the estimated parameters of player 1: $\hat{\bm{\alpha}}_1 \approx \mat{ 1.415 & 3.592 }^\top$ and $\hat{\bm{\beta}}_1 \approx \mat{ -3.568 & -5.509 & 3.301 }^\top$. Now, in line with Lemma~\ref{lemma:costFct_error_offline}, the FNE $\hat{\bm{\mu}}^*$ corresponding to the estimated parameters differs from the GT FNE. This is illustrated in Fig.~\ref{fig:comparisonGTvsIDGQerror} (left subfigure) where the estimated trajectories do not even show similar trends to the GT ones and by high NSAE values: $\delta^x \approx 1810.7$ and $\delta^u \approx 2079.7$. The overall computation time was $1.88\,\text{s}$.

Due to the missing PE of the regression vector~$\bm{m}_{\hjb_1}(t)$ the boundedness of the estimated parameters $\hat{\bm{\alpha}}_1, \hat{\bm{\beta}}_1$ cannot be guaranteed theoretically. Although we observe in the right subfigure of Fig.~\ref{fig:learningTrajPlayer1} bounded learning trajectories, they oscillate in a broad range where not all parameter values even guarantee solvability of the DG. For example, near $t=10\,\text{s}$ negative values for $\hat{\alpha}_{1,1}$ and $\hat{\alpha}_{1,2}$ occur and a negative definite matrix $\hat{\bm{R}}_{11}$ would result. By choosing the estimated parameters $\hat{\bm{\alpha}}_1 \approx \mat{ 0.875 & 1.509 }^\top$ and $\hat{\bm{\beta}}_1 \approx \mat{ -1.572 & -3.837 & 1.825 }^\top$ from the residual set at $t=8.957\,\text{s}$, the DG is solvable and with the resulting FNE comparable errors to the GT trajectories as in the offline case follow (cf. right subfigure in Fig.~\ref{fig:comparisonGTvsIDGQerror}): $\delta^x \approx 1907.0$ and $\delta^u \approx 2686.8$.

\section{Conclusion} \label{sec:conclusion}
In this work, we propose new offline and online IDG methods for nonlinear systems based on the coupled HJB equations. With the offline method, we are able to compute the most general solution set to analyze the non-uniqueness of IDG solutions, i.e. the set of all equivalent cost function parameters that yield the GT trajectories. Furthermore, the online approach is the first online learning IDG method for nonlinear systems and it is guaranteed to converge to one element of the offline solution set. Finally, the separate analysis of erroneous approximations of the cost and value functions, i.e. that assumed cost or value function approximation structures lead to remaining errors to the true functions, reveals that it is crucial to know the cost function structures for assumed value function structures such that fulfillment of the HJB equations is guaranteed. Our new IDG methods build a theoretical framework for various extensions in future work, e.g. model-free versions of our algorithms or their application to tracking or shared control problems.


\bibliographystyle{IEEEtran}

\bibliography{References3}

\begin{IEEEbiography}[{\includegraphics[width=1in,height=1.25in,clip,keepaspectratio]{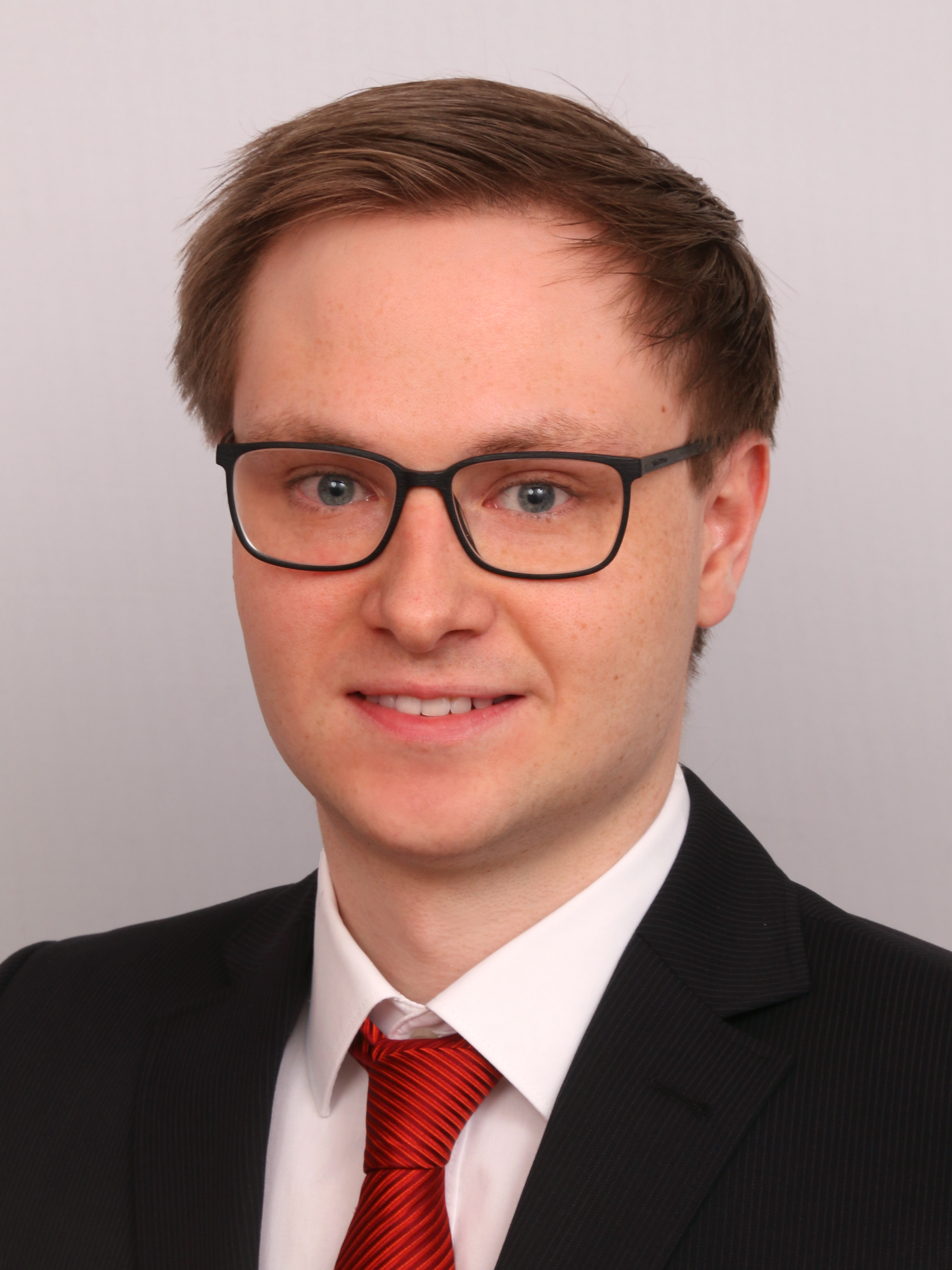}}]{Philipp Karg}
 	studied Electrical Engineering and Information Technologies at the Karlsruhe Institute of Technology (KIT), Karlsruhe, Germany, where he received the B.Sc. degree in 2016 and the M.Sc. degree in 2019. Since 2019, he is a Ph.D. candidate at the Institute of Control Systems (IRS) at the KIT. His current research interests include Inverse Optimal Control and Inverse Reinforcement Learning in single- and multi-agent settings.
\end{IEEEbiography}

\begin{IEEEbiography}[{\includegraphics[width=1in,height=1.25in,clip,keepaspectratio]{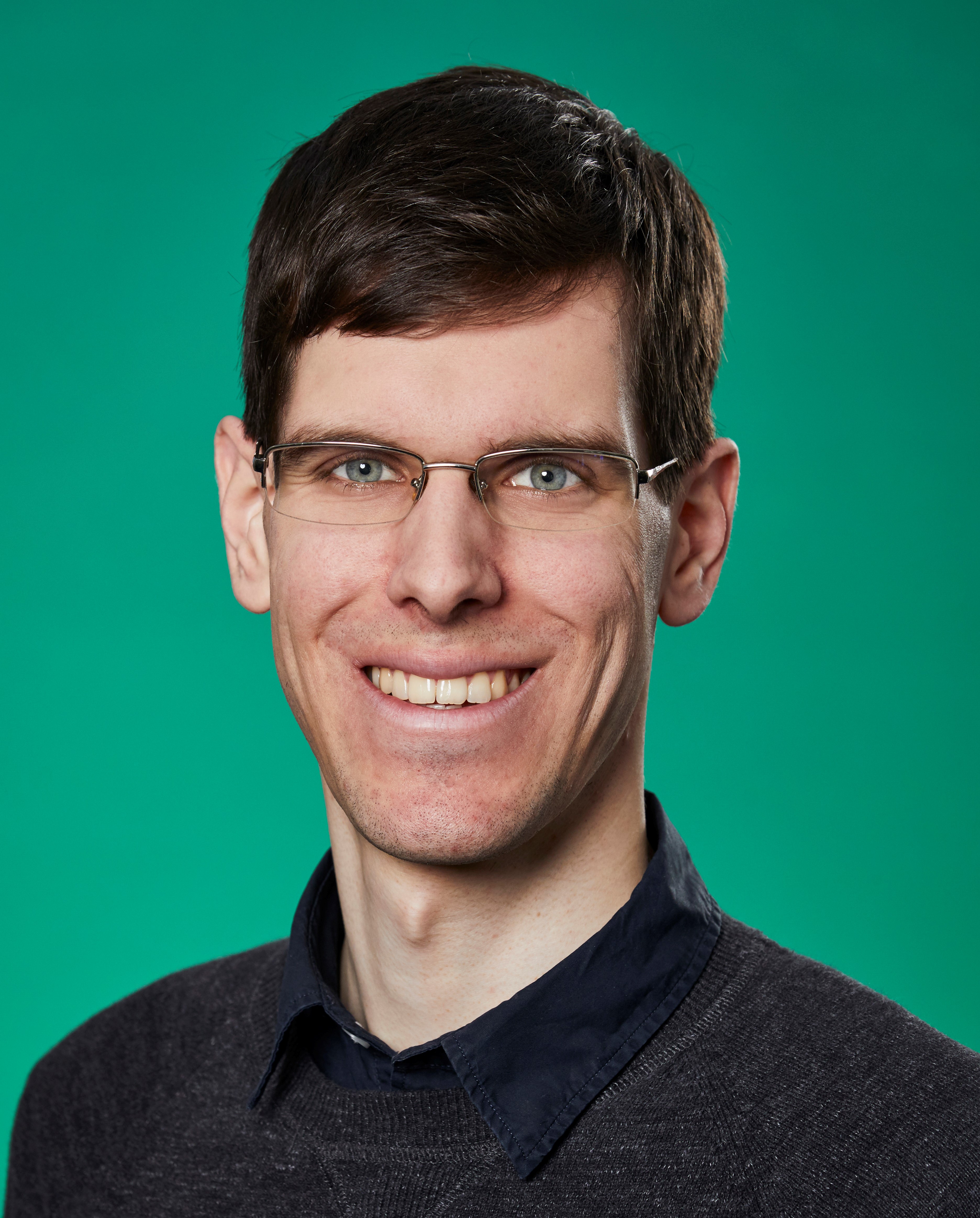}}]{Balint Varga}
	received his B.Sc. in Mechatronics from the Technical University of Budapest, Hungary, in 2016. He earned an M.Sc. in Mechanical Engineering in 2017 and a Dr.-Ing. in Electrical Engineering and Information Technologies in 2023, both from the KIT, Germany. 
	In 2023, he became the head of the research group on Cooperative Systems. His research interests include modeling human-machine interaction and designing shared-control concepts for various applications. He is a member of the IEEE SMC Society and leads the Technical Committee on Shared Control. He serves as an Associate Editor for the IEEE Transactions on Human-Machine Systems and the International Journal of Robotics and Automation.
\end{IEEEbiography}

\begin{IEEEbiography}[{\includegraphics[width=1in,height=1.25in,clip,keepaspectratio]{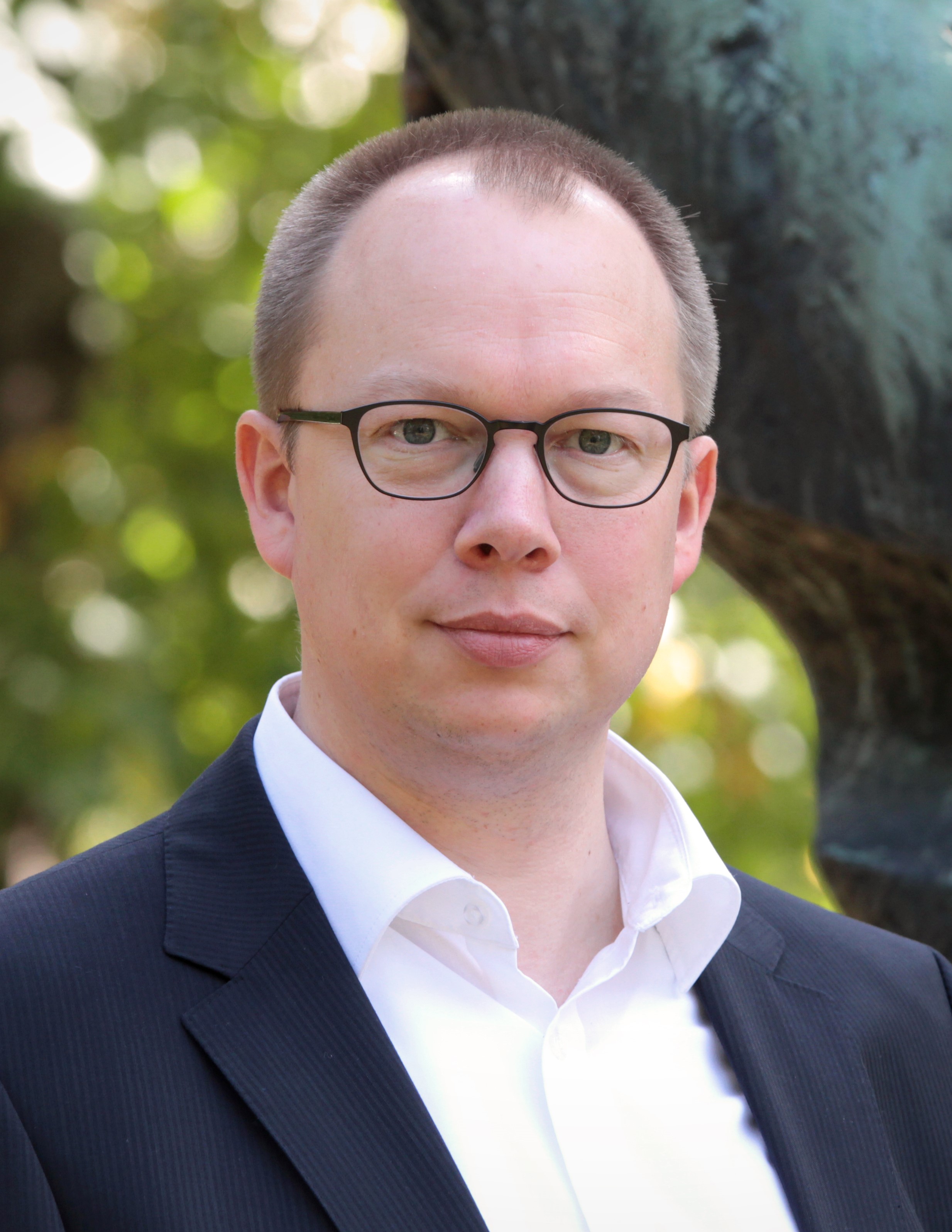}}]{Sören Hohmann}
	studied Electrical Engineering at the Technische Universität Braunschweig, University of Karlsruhe and école nationale supérieure d’électricité et de mécanique Nancy. He received his diploma degree and Ph.D. degree at the University of Karlsruhe. Afterwards, he worked for BMW in different executive positions. Since 2010, he is the head of the IRS at the KIT. In addition, he is a director of the FZI Research Center for Information Technology. His research interests are cooperative control systems, control of energy systems and functional safety control. He is senior member of the IEEE and chair of the IFAC technical committee on human-machine systems.
\end{IEEEbiography}

\end{document}